\documentclass[12pt]{amsart}

\usepackage{amssymb,amsbsy}
\usepackage{latexsym,array}
\usepackage{verbatim,color}
\usepackage{fullpage,euscript}
\usepackage{xcolor}
\usepackage{tikz}
\usetikzlibrary{arrows,shapes,trees}
\usepackage[colorlinks=true,linkcolor=blue,urlcolor=my_color,citecolor=magenta]{hyperref}

\definecolor{my_color}{rgb}{0,0.5,0.5}
\definecolor{MIXT}{rgb}{0.8,0.5,0.2}
\definecolor{mixt}{rgb}{0.5,0.3,0.2}
\definecolor{sin}{rgb}{0,0.5,0.5}
\definecolor{darkblue}{rgb}{0,0.1,0.8}
\definecolor{redi}{rgb}{0.5,0,0.4}

\numberwithin{equation}{section}

\makeatletter
\@namedef{subjclassname@2020}{\textup{2020} Mathematics Subject Classification}
\makeatother

\input {cyracc.def}

\newtheorem{thm}{Theorem}[section]
\newtheorem{prop}[thm]{Proposition}
\newtheorem{lm}[thm]{Lemma}
\newtheorem{cl}[thm]{Corollary}

\theoremstyle{remark}
\newtheorem{rmk}[thm]{Remark}
\newtheorem{ex}[thm]{Example} 

\theoremstyle{definition}
\newtheorem{df}{Definition}

\newcommand {\g}{{\mathfrak g}}
\newcommand {\gl}{{\mathfrak{gl}}}

\newcommand {\h}{{\mathfrak h}}

\newcommand {\me}{{\mathfrak m}}

\newcommand {\q}{{\mathfrak q}}
\newcommand {\rr}{{\mathfrak r}}

\newcommand {\z}{{\mathfrak z}}

\newcommand {\gln}{{\mathfrak{gl}}_n}

\newcommand {\sln}{{\mathfrak{sl}}_n}
\newcommand {\slv}{{\mathfrak{sl}}(\gV)}

\newcommand {\spv}{{\mathfrak{sp}}(\gV)}

\newcommand {\sov}{{\mathfrak{so}}(\gV)}

\newcommand {\son}{{\mathfrak {so}}_{n}}

\newcommand {\eus}{\EuScript}

\newcommand {\gC}{{\eus C}}

\newcommand {\gS}{{\eus S}}

\newcommand {\sT}{{\mathsf T}}
\newcommand {\gV}{{\eus V}}
\newcommand {\gZ}{{\eus Z}}

\newcommand {\ap}{\alpha}

\newcommand {\lb}{\lambda}
\newcommand {\vth}{\vartheta}
\newcommand {\vp}{\varphi}

\newcommand {\tfg}{\tilde\g}

\newcommand {\tfq}{\tilde\q}
\newcommand {\tx}{\tilde x}
\newcommand {\ty}{\tilde y}

\newcommand {\wD}{\widehat D}

\newcommand {\ca}{{\mathcal A}}

\newcommand {\gc}{{\mathcal C}}

\newcommand {\co}{{\mathcal O}}
\newcommand {\cR}{{\mathcal R}}
\newcommand {\ct}{{\mathcal T}}

\newcommand {\cz}{{\mathcal Z}}

\newcommand {\BP}{{\mathbb P}}

\newcommand {\BU}{{\mathbb U}}

\newcommand {\BZ}{{\mathbb Z}}
\newcommand {\BN}{{\mathbb N}}

\newcommand {\ad}{{\mathrm{ad}}}

\newcommand {\ind}{{\mathrm{ind\,}}}
\newcommand {\Lie}{{\mathsf{Lie\,}}}
\newcommand {\Ker}{\operatorname{Ker}}
\newcommand {\Ima}{{\mathrm{Im\,}}}

\newcommand {\rk}{{\mathsf{rk\,}}}

\newcommand {\tri}{\mathfrak{sl}_2}
\newcommand {\GR}[2]{{\textrm{{\sf\bfseries #1}}}_{#2}}

\newcommand {\un}{\underline}
\newcommand {\blb}{\boldsymbol{\lb}}

\newcommand {\bd}{{\boldsymbol{d}}}
\newcommand {\bv}{\boldsymbol{v}}

\newcommand {\PC}{{\sf PC}}

\newcommand {\beq}{\begin{equation}}
\newcommand {\eeq}{\end{equation}}

\renewcommand{\le}{\leqslant}
\renewcommand{\ge}{\geqslant}
\renewcommand{\lg}{\langle}
\newcommand{\rg}{\rangle}

\newcommand{\wrt}{{w.r.t.}}

\newcommand {\bbk}{\Bbbk}

\begin{document}
\setlength{\parskip}{3pt plus 2pt minus 0pt}
\hfill {\scriptsize March 31, 2026}
\vskip1ex

\title[Near-derivations]
{Near-derivations and their applications to Lie algebras}
\author[D.\,Panyushev]{Dmitri I. Panyushev}
\address[D.P.]{Independent University of Moscow, 119002 Moscow, Russia}
\email{panyush@mccme.ru}
\author[O.\,Yakimova]{Oksana S.~Yakimova}
\address[O.Y.]{Institut f\"ur Mathematik, Friedrich-Schiller-Universit\"at Jena,  07737 Jena,
Germany}
\email{oksana.yakimova@uni-jena.de}
\begin{abstract}
E.B.\,Vinberg's concept of {\it quasi-derivations\/} of algebras is extended to a broader framework of 
{\it near-derivations}. This deepens connections between Poisson geometry and Lie theory. Although basic 
results apply to arbitrary algebras, our substantial applications  concern the Poisson algebra 
$(\gS(\q),\{\ ,\,\})$ of a Lie algebra $\q$. We develop a method for obtaining quasi-derivations via the use 
of squares of derivations, which allows us to provide quasi-derivations of the simple Lie algebras.
It is shown that {\bf (1)} a near-derivation $D$ of $(\gS(\q),\{\ ,\,\})$ 
yields a pencil of compatible Poisson brackets on $\q^*$ and  {\bf (2)} using $D$ one may naturally 
construct a Poisson-commutative subalgebra of $\gS(\q)$.  A special attention is given to near-derivations 
of $(\gS(\q),\{\ ,\,\})$ induced from near-derivations of $\q$. This provides some old and new families of 
compatible Poisson brackets. We also compare properties of near-derivations of $\q$ and Nijenhuis 
operators in $\gl(\q)$.       
\end{abstract}
\maketitle

\section{Introduction}            \label{sect:intro}
In this article, we generalise Vinberg's theory of quasi-derivations of algebras~\cite{vi95} using a more general notion of near-derivations. Although this theory applies to arbitrary algebras (not necessarily
finite-dimensional), main applications concern the Poisson algebras of finite-dimensional Lie algebras
and Lie algebras themselves.

For a $\bbk$-vector space $\ca$ with a bilinear operation $\psi:\ca\times\ca\to \ca$, let 
$\sT\in \ca\otimes\ca^*\otimes\ca^*=:\Pi_2^1(\ca)$ denote the corresponding tensor of structure constants.
Write $(\ca,\sT)$ or $(\ca,\psi)$ for this algebra. Let $\rho$ be the natural representation of the Lie algebra 
$\gl(\ca)$ in $\Pi^1_2(\ca)$. If $D\in\gl(\ca)$, then the bilinear operation corresponding to 
$\rho(D){\cdot}\sT$ is said to be the {\it derived operation of $\psi$ \wrt\ to} $D$. It is denoted by $\psi'_D$, and one has  
\beq   \label{eq:derived}
    \psi'_D(x,y)=D(\psi(x,y))- \psi(Dx,y)-\psi(x,Dy).
\eeq
Iterating this procedure, one obtains formulae for higher derived operations of $\psi$. 
By~\eqref{eq:derived}, a {\it derivation\/} of $(\ca,\sT)$ is a transformation $D\in\gl(\ca)$ such that 
$\rho(D){\cdot}T=0$, and the Lie algebra of all derivations is denoted by ${\sf Der}(\ca,\psi)$. Then

\textbullet \ \ a {\it quasi-derivation\/} of $(\ca,\sT)$ is $D\in\gl(\ca)$ such that $\rho(D)^2{\cdot}T=0$;
\\ \indent
\textbullet \ \ a {\it near-derivation\/} of $(\ca,\sT)$ is $D\in\gl(\ca)$ such that $\rho(D)^2{\cdot}T\in
\lg T, \rho(D){\cdot}T\rg$. 
\\[.5ex]
To avoid trivialities, the space $\BU=\lg T, \rho(D){\cdot}T\rg\subset\Pi_2^1(\ca)$ is assumed to be 
$2$-dimen\-si\-onal. We also assume that $D$ is {\it locally finite-dimensional}, i.e.,
$\dim\lg D^n(f)\mid n=0,1,2,\dots\rg<\infty$ for all $f\in \ca$. Our main observation is that, for a locally 
finite-dimensional near-derivation $D$,  there is a dense open subset 
$\Omega\subset\BU$ that consists of algebras isomorphic to $(\ca,\sT)$. Therefore, if 
$\tilde\sT\in\BU\setminus\Omega$, then $(\ca,\tilde\sT)$ is a degeneration of $(\ca,\sT)$. 

\textbullet \ For any $D$ and $d\in {\sf Der}(\ca,\psi)$, we give a formula for the second derived operation
$\psi''_{D+d}$. This implies that if $[d,D]=0$, then $D$ is a near-derivation if and only if $D+d$ is.
\\ \indent
\textbullet \ We also give, for any $n$, a formula for the $n$-th derived operation of $\psi$ \wrt\ $D:=d^2$. 
As a corollary, we notice that if $d^4=0$ and $\psi\vert_{\Ima D}=0$, then $D$ is a quasi-derivation.

Let $\q$ be a finite-dimensional Lie algebra and $\gS(\q)$ the symmetric algebra of $\q$ with the Poisson 
bracket $\{\ ,\,\}$. Two Lie brackets $[\ ,\, ]$ and $[\ ,\,]'$ on the vector space $\q$ are said to  be 
{\it compatible}, if $[\ ,\,]+[\ ,\,]'$ is again a Lie bracket (equivalently, if $a[\ ,\,]+b[\ ,\,]'$ is a Lie bracket for 
all $a,b\in\bbk$.) The same terminology applies to a pair of Poisson brackets in $\gS(\q)$. If 
$(\ca,\psi)=(\gS(\q),\{\ ,\,\})$ and $\wD$ is a near-derivation of $(\gS(\q), \{\ ,\,\})$, then  the `main 
observation' above implies that the Poisson brackets $\{\ ,\,\}$ and $\{\ ,\,\}'_{\wD}$ are compatible. 
Moreover, we give a recipe for constructing a Poisson-commutative subalgebra 
$\cz_{\wD}\subset \gS(\q)$. Namely, if $\cz\gS(\q)$ is the centre of $(\gS(\q), \{\ ,\,\})$ (=\, the Poisson 
centre of $\gS(\q)$), then $\cz_{\wD}$ is generated by $\{{\wD}^n (f)\mid f\in\cz\gS(\q), \ n=0,1,2,\dots\}$. 
The pencil of compatible Poisson brackets $\eus P=\{\ap\{\ ,\,\}+\beta\{\ ,\,\}'_{\wD} \mid \ap,\beta\in\bbk\}$ 
can be identified with $\BU$ and it is also shown that $\cz_{\wD}$ coincides with the subalgebra of 
$\gS(\q)$ generated by the Poisson centres associated with the generic elements of $\BU$.

In order to obtain a near-derivation of $(\gS(\q), \{\ ,\,\})$, one may begin with some $D\in\gl(\q)$ and 
extend it to the linear transformation $\widehat D$ of $\gS(\q)$ using the Leibniz rule, i.e., 
$\widehat D(xy)=(Dx)y+x(Dy)$, etc. Then $\widehat D$ is locally finite-dimensional and it is also a
derivation of $(\gS(\q),{\cdot})$. By definition, the derived operation in $\q$ \wrt\ $D$ is 
\[
    [x,y]'_D=D([x,y])-[Dx,y]-[x,Dy]
\]
and then $\{\ ,\,\}'_{\widehat D}$ is obtained from $[\ ,\,]'_D$ via the use of the Leibniz rule. Hence 
\\[.5ex]
\centerline{ $D$ is a near-derivation of $(\q, [\ ,\,])$ $\Leftrightarrow$ 
$\widehat D$ is a near-derivation of $(\gS(\q), \{\ ,\,\})$.} 
\\[.5ex] 
If this is the case, then $[\ ,\,]$ and $[\ ,\,]'_D$ are compatible Lie brackets, etc. This means that we can mainly work with near-derivations of a Lie algebra $\q$. 

For any $\q$, we describe a practical method for obtaining quasi-derivations in terms of $\BZ$-gradings. 
Moreover, if $\g$ is a reductive Lie algebra, then there is the following explicit result. If $e\in\g$ and 
$(\ad\,e)^3=0$, then $D=(\ad\,e)^2$ is a quasi-derivation. Hence $[\ ,\,]$ and $[\ ,\,]'_{D}$ are 
compatible Lie brackets. Here $[\ ,\,]'_{D}$ is the only non-generic line in the pencil $\eus P$, and 
we prove that the index of the Lie algebra $\g_D:=(\g, [\ ,\,]'_D)$ equals the dimension of the centraliser 
of $e$ in $\g$. It is also shown that $\g_D$ is two-step nilpotent. We show that previously studied 
Poisson-commutative subalgebras of $\gS(\g)$ associated with periodic gradings of $\g$~\cite{period1}
or $2$-splittings~\cite{bn} are actually related to semisimple near-derivations of $\g$.

Then we discuss an approach to semisimple
near-derivations of $\q$ via $1$-parameter deformations of Lie brackets and compare properties of 
Nijenhuis operators and near-derivations. If $D\in\gl(\q)$, then we express the Nijenhuis torsion of $D$
via the first and second derived operations of $[\ ,\,]$ \wrt\ $D$. In particular, it is shown that if $D^2=0$, then the conditions of being Nijenhuis or quasi-derivation are equivalent.

The structure of the paper is as follows. Results of E.\,Vinberg on quasi-derivations of arbitrary algebras 
are discussed in Section~\ref{sect:quasi}. In Section~\ref{sect:near}, we obtain our main results on 
near-derivations of arbitrary algebras and show that Vinberg's method can be adapted to our setting. 
In Section~\ref{sect:Near-Pois}, we consider in more details near-derivations of the Poisson algebras $(\gS(\q), \{\ ,\,\})$.
Section~\ref{sect:near-Lie} is devoted to some old and new examples of compatible Lie brackets related
to semisimple near-derivations.
In Section~\ref{sect:2term-deformations}, we study $1$-parameter deformations of Lie brackets, 
while Section~\ref{sect:quasi-Z} is devoted to quasi-derivations of Lie algebras and their connections to 
nilpotent elements and $\BZ$-gradings, especially for reductive Lie algebras. Finally, in Sections~\ref{sect:near-Nij} and \ref{sect:assoc}, 
we consider Nijenhuis operators in $\gl(\q)$, Nijenhuis torsion, and their connection with near-derivations.
 
{\it\bfseries Notation.} Throughout, $\q$ is an algebraic Lie algebra with $\q=\Lie Q$, while the letter $\g$ 
is reserved for a reductive Lie algebra. A direct sum of Lie algebras is denoted by `$\dotplus$'.

\section{Quasi-derivations of algebras} 
\label{sect:quasi}
\noindent
Here we recall Vinberg's results on quasi-derivations~\cite{vi95}. 

Let $\gV$ be a vector space over $\bbk$ (not necessarily finite-dimensional). The structure of an algebra 
on $\gV$ is an element of the vector space $\Pi^1_2(\gV)=\gV\otimes \gV^*\otimes \gV^*$ of tensors of 
type $(1,2)$ on $\gV$. Then $(\gV,\sT)$ denotes the algebra determined by a tensor 
$\sT\in \Pi^1_2(\gV )$. 

Let $\rho$ be the natural representation of the Lie algebra $\gl(\gV)$ in $\Pi^1_2(\eus V)$. For 
$D\in \gl(\gV)$, the tensor $\rho(D){\cdot}\sT$ represents another bilinear operation on $\gV$. Following 
Vinberg, we say that $\rho(D){\cdot}\sT$ represents the {\it derived operation} (\wrt\ $D$) of the initial 
operation. Let $\psi$ denote the bilinear operation on $\gV$ corresponding to $\sT$ and $\psi'_D$ the 
derived operation \wrt\ $D$. Then the definition of $\rho$ means that, for all $x,y\in \gV$, one has
\beq    \label{eq:1deriv}
      \psi'_D(x,y)=D(\psi(x,y))-\psi(Dx,y)-\psi(x,Dy) .
\eeq
This procedure can be iterated, so that $\rho(D)^n{\cdot}\sT$ is the tensor of the {\it $n$-th derived 
operation}. These derived operations are denoted by $\psi'_D, \psi''_D, \dots,\psi^{(n)}_D$, etc.

In this setting, a {\it derivation\/} of the algebra $(\gV,\sT)$ is a transformation $D\in\gl(\gV)$ such that  
$\rho(D){\cdot}\sT=0$, i.e., $\psi'_D(x,y)=0$ for all $x,y\in\gV$.

\begin{df}[Vinberg]
A {\it quasi-derivation} of $(\gV,\sT)$ is a transformation $D\in \gl(\gV)$ such that 
$\rho(D)^2{\cdot}\sT=0$.
\end{df}

Therefore, $D$ is a quasi-derivation of $(\gV,\sT)$ {\sl if and only if\/} $D$ is a derivation of
$(\gV, \rho(D){\cdot}\sT)$ {\sl if and only if\/} $\psi''_D(x,y)=0$ for all $x,y\in\gV$.
Iterating procedure \eqref{eq:1deriv} yields the explicit formula for the second derived operation
$(\psi'_D)'_D=\psi''_D$.  We give it below in the proof of Theorem~\ref{thm:near2}.
\\ \indent
In~\cite{vi95}, Vinberg establishes two general properties of quasi-derivations.  

\begin{thm}   \label{thm:v2}
Suppose that $D$ is a quasi-derivation of $(\gV,\sT)$. If $x,y\in \gV$ have the property that
$\psi(D^nx,y)=\psi(x,D^ny)=0$ for all $n\in \BN$, then $\psi(D^kx,D^ly)=0$ for all $k,l$.
\end{thm}

\begin{thm}   \label{thm:v1}
If $D$ is a quasi-derivation of $(\gV,\sT)$, then the algebra $(\gV,\rho(D){\cdot}\sT)$ satisfies all the 
identities that $(\gV,\sT)$ satisfies.
\end{thm}

Vinberg's proof of Theorem~\ref{thm:v1} relies on the fact that the 2-dimensional vector space 
$\BU:=\lg \sT, \rho(D){\cdot}\sT\rg$ in $\Pi^1_2(\gV)$ is $D$-stable. (We assume that 
$\rho(D){\cdot}\sT\ne 0$.) Note that $\rho(D)\vert_\BU$ is nilpotent and its matrix \wrt\ the basis 
$(\sT, \rho(D){\cdot}\sT)$ equals {\small $\begin{pmatrix} 0 & 0 \\ 1 & 0\end{pmatrix}$}.

We say more about the proofs in Section~\ref{sect:near}, when these two theorems will be placed in a more general setting.
\begin{ex}          \label{ex:vinb}
Vinberg's main motivation for considering quasi-derivations came from the following situation. Let 
$\gS(\q)=\bigoplus_{p=0}^\infty\gS^p\q$ be the symmetric algebra of a finite-dimensional Lie algebra $\q$. It can be identified with 
the graded ring of polynomial functions on the dual space $\q^*$. We primarily consider $\gS(\q)$ as a 
commutative-associative algebra with respect to the usual multiplication `${\cdot}$'. But it also has an 
additional operation, the Poisson bracket $\{\ ,\,\}$, coming from the commutator in the enveloping 
algebra of $\q$. So, one has in $\gS(\q)$ the usual multiplication $(f,g)\mapsto f{\cdot}g$ and the Poisson 
bracket $(f,g)\mapsto \{f,g\}$. Recall that $\{\ ,\,\}$ satisfies the Jacobi identity and these two operations 
in $\gS(\q)$ obey the {\it Leibniz rule}:
\[
                   \{f{\cdot}g, h\}=\{f,h\}{\cdot}g+f{\cdot}\{g,h\}
\]
for $f,g,h\in \gS(\q)$. If $f\in\gS^n\q$ and $g\in\gS^m\q$, then $f{\cdot}g\in \gS^{n+m}\q$ and
$\{f,g\}\in \gS^{n+m-1}\q$. A usual subalgebra $\gc$ of $\gS(\q)$ is said to be {\it Poisson-commutative}, 
if $\{f,h\}=0$ for all $f,h\in\gc$.

For $\gamma\in\q^*$, let $D_\gamma: \gS(\q)\to \gS(\q)$ be the derivation of $(\gS(\q), {\cdot})$ defined 
on $\q=\gS^1\q$ by $D_\gamma(x):=\gamma(x)\in \bbk$. Then $D_\gamma$ is the directional derivative of $\gS(\q)$ 
relative to $\gamma$. If $f\in\gS(\q)$ is regarded as polynomial on $\q^*$ and $\xi\in\q^*$, then 
the polynomial $D_\gamma f$ is defined as
\[
   D_\gamma f(\xi)=\lim_{t\to 0}\frac{f(\xi+t\gamma)-f(\xi)}{t} .
\]
But $D_\gamma$ is {\bf not} a derivation with respect to the Poisson bracket, and it is shown in~\cite{vi95} 
that $D_\gamma$ is a quasi-derivation of $(\gS(\q), \{\ ,\,\})$. Then it follows from Theorem~\ref{thm:v1} 
that the derived operation $\{\ ,\,\}_{\gamma}:=\{\ ,\,\}'_{D_\gamma}$ is again a Poisson bracket in 
$\gS(\q)$. This is a so-called Poisson bracket ``with frozen argument'', see e.g.~\cite[\S\,1.1]{bols}. If $f,g\in \gS(\q)$ are regarded as
polynomial functions on the dual space $\q^*$, then the function $ \{f,g\}_\gamma$ is defined by
\[
         \{f,g\}_\gamma(\xi)=\gamma ( [\textsl{d}f(\xi),\textsl{d}g(\xi)]),
\]
where $\textsl{d}f(\xi)\in \q$ is the differential of $f$ at $\xi\in\q^*$. In particular, $\{x,y\}_\gamma(\xi)=
\gamma([x,y])$ for  $x,y\in\q$ and any $\xi\in\q^*$.

Let $\cz\gS(\q)$ be the {\it Poisson centre\/} of $(\gS(\q), \{\ ,\,\})$, i.e.,
\[
       \cz\gS(\q)=\{f\in \gS(\q)\mid \{f,h\}=0 \text{ for all } h\in \gS(\q)\}.
\]
The Leibniz rule shows that $\cz\gS(\q)$ is a subalgebra of $\gS(\q)$ \wrt\ the usual multiplication.
It then follows from Theorem~\ref{thm:v2} that the subalgebra of $(\gS(\q),{\cdot})$
generated by the elements $D_\gamma^n(z)$, where $z\in \cz\gS(\q)$ and  $n\in\BN$, is Poisson-commutative (=\,\PC). This subalgebra is 
nothing but the {\it Mishchenko--Fomenko subalgebra\/} corresponding to $\gamma$, see e.g.~\cite{bols,mf,vi90}.
\end{ex}

Vinberg's approach via quasi-derivations of $\gS(\q)$ provides a bird's-eye view towards the "argument 
shift method" and constructing the Mishchenko--Fomenko subalgebras.
\begin{rmk}       \label{rmk:v1}
Although this is not stated directly in~\cite{vi95}, Vinberg's proof of Theorem~\ref{thm:v1} implies that in
the setting of Example~\ref{ex:vinb} the Poisson brackets $\{\ ,\,\}$ and $\{\ ,\,\}_\gamma$ are 
compatible. We will discuss this subject in Section~\ref{sect:near}.
\end{rmk}

For the sake of completeness, we accurately state a curious observation on derived operations, which is 
easily checked via direct calculations.

\begin{prop}[Vinberg]    \label{prop:derived-L}
Let $\{\ ,\,\}$ be a skew-symmetric operation in $\gS(\q)$ satisfying the Leibniz rule. If $D$ is an arbitrary derivation of 
$(\gS(\q), {\cdot})$, then $\{\ ,\,\}'_D$ also satisfies the Leibniz rule.
\end{prop}

The point is that the Jacobi identity for $\{\ ,\,\}$ or $\{\ ,\,\}'_D$ is irrelevant here. This result is 
mentioned without details in~\cite[p.\,188]{vi95}.
\\ \indent
It is worth stressing that the concept of quasi-derivations applies to arbitrary algebras, and 
\cite{vi95} contains also interesting observations related to associative algebras and superalgebras.

\section{Near-derivations of algebras} 
\label{sect:near}

\noindent
Let $(\ca,\sT)$ be an algebra over $\bbk$ determined by the tensor of structure constants 
$\sT\in\Pi^1_2(\ca)$. As above, $\psi$ is the bilinear operation in $\ca$ corresponding to  
$\sT$, and we also write $(\ca,\psi)$ for this algebra.  Our new notion is the following.

\begin{df}                   \label{def:near}
A {\it near-derivation\/} of the algebra $(\ca,\sT)$ is a transformation $D\in\gl(\ca)$ such that 
$\rho(D)^2{\cdot}\sT\in\lg \sT, \rho(D){\cdot}\sT\rg\subset\Pi^1_2(\ca)$.
\end{df}

First, we obtain some general properties of derived operations.
\begin{prop}     \label{prop:1}
If $d$ is a derivation of $(\ca,\psi)$ and $D=d^2$, then the $n$-th derived operation of $\psi$
\wrt\ $D$ is given by
\[
           \psi^{(n)}_D(x,y)=2^n {\cdot} \psi\bigl( d^n(x),d^n(y)\bigr) .
\]
\end{prop}
\begin{proof}   (a)
By definition, if $n=1$, then $\psi'_D(x,y)=d^2(\psi(x,y))-\psi(d^2(x),y)-\psi(x,d^2(y))$. Since
$d\in {\sf Der}(\ca,\psi)$, we have
\[
     d^2(\psi(x,y))= \psi(d^2(x),y)+2\psi(d(x),d(y))+   \psi(x,d^2(y)).
\] 
(b) \ Since  $\psi'_D(x,y)=2\psi(\tx,\ty)$, where $\tx=d(x)$ and $\ty=d(y)$, we may argue by induction on
$n$ using $\tx$ and $\ty$ in place of $x$ and $y$.
\end{proof}

For $n=2$, we obtain $\psi''_D(x,y)=4{\cdot}\psi(Dx, Dy)$. This readily implies the following.
\begin{cl}     \label{cor:2}
If $d\in {\sf Der}(\ca,\psi)$ and $D=d^2$, then $D$ is a quasi-derivation of $(\ca,\psi)$ if and only if\/ 
$\psi\vert_{\Ima D}\equiv 0$.
\end{cl}
As we shall see in Section~\ref{sect:quasi-Z}, this allows us to construct explicitly quasi-derivations of
(finite-dimensional) reductive Lie algebras.

\begin{prop}          \label{prop:2}
If $D\in\gl(\ca)$ and $d\in {\sf Der}(\ca,\psi)$, then  $\psi'_{D+d}=\psi'_D$
and  $\psi''_{D+d}=\psi''_D+\psi'_{[d,D]}$.
\end{prop}
\begin{proof}
1$^o$. We have $\psi'_{D+d}=\psi'_D+\psi'_d$ and $\psi'_d=0$. \\
2$^o$. Next, $\psi''_{D+d}=(\psi'_D)'_{D+d}=\psi''_D+(\psi'_D)'_d$. Hence we have to compute the last term. By definition,
$\psi'_D=\psi_1+\psi_2+\psi_3$, where
\[
    \psi_1(x,y)=D(\psi(x,y)), \ \psi_2(x,y)=-\psi(Dx,y), \ \psi_3(x,y)=-\psi(x,Dy) .
\]
Gathering together $(\psi_i)'_d$ for $i=1,2,3$, we obtain
\begin{multline}
((\psi'_D)'_d)(x,y)=dD(\psi(x,y))-D\psi(dx,y)- D\psi(x,dy)\\
   -d\psi(Dx,y)+\psi(Ddx,y)+\psi(Dx,dy) \\   
   -d\psi(x,Dy)+\psi(dx,Dy)+\psi(x,Ddy) .
\end{multline}
Using the fact that $d\in {\sf Der}(\ca,\psi)$, we transform this into
\[
  (dD-Dd)\psi(x,y)-\psi([d,D]x,y)-\psi(x,[d,D]y)=\psi'_{[d,D]}(x,y) .   \qedhere
\]
\end{proof}
\begin{cl}
 If $D$ is a quasi-derivation (or near-derivation) and $[d,D]=0$, then so is $D+d$.
\end{cl}
As any quasi-derivation is a near-derivation, our next goal is to prove that Theorems~\ref{thm:v2} 
and \ref{thm:v1} remain valid for the near-derivations. 
 
\begin{thm}                \label{thm:near2}
Let $D$ be a near-derivation of $(\ca,\psi)$. Suppose that the elements $x,y\in \ca$ have the property that
$\psi(D^ix,y)=\psi(x,D^iy)=0$ for all $i\le n$. Then $\psi(D^ix,D^jy)=0$ for all $i,j$ with $i+j\le n$.
\end{thm}
\begin{proof}
1) If $n=1$, then there is nothing to prove. For $n=2$, we need the explicit formula for the second derived operation $\psi''_D$. Using~\eqref{eq:1deriv}, we obtain
\begin{multline}    \label{eq:2-proizv}
   \psi''_D(x,y)=\psi(D^2x,y){+}2\psi(Dx,Dy){+}\psi(x,D^2y){+}D^2(\psi(x,y)){-}  \\
   2D\bigl(\psi(Dx,y){+}\psi(x,Dy)\bigr).
\end{multline}
If $\psi(x,y)=\psi(x,Dy)=\psi(Dx,y)=0$, then $\psi'_D(x,y)=0$; and if $\psi''_D\in\lg\psi, \psi'_D\rg$, then
\[
   0=\psi''_D(x,y)=\psi(D^2x,y){+}2\psi(Dx,Dy){+}\psi(x,D^2y).
\]
Hence the assumptions $\psi(D^2x,y)=\psi(x,D^2y)=0$ imply that $\psi(Dx,Dy)=0$.

2) For $n>2$, we argue by induction on $n$. Suppose that the assertion holds for all $i,j$ with $i+j<n$.
For $k+l=n$ with $k,l>0$, set
$\xi=D^{k-1}x$ and $\eta=D^{l-1}y$. By the induction assumption, $\psi(\xi,\eta)=
\psi'_D(\xi,\eta)=0$.
Using the formula for $\psi''_D(\xi,\eta)$, we get
\begin{multline*}
   \psi''_D(\xi,\eta)=\psi(D^2\xi,\eta){+}2\psi(D\xi,D\eta){+}\psi(\xi,D^2\eta)=\\
   \psi(D^{k+1}x,D^{l-1}y)+ 2\psi(D^{k}x,D^{l}y)+ \psi(D^{k-1}x,D^{l+1}y)=0 .
\end{multline*}
Doing this for all such pairs $(k,l)$, we conclude that the numbers $z_j=(-1)^j\psi(D^jx,D^{n-j}y)$,
$j=0,1,\dots,n$, form an arithmetic progression. Since the extreme terms $z_0$ and $z_n$ are zero, all of them are equal to zero.
\end{proof}

{\it\bfseries Conventions.} {\sf (i)} It is natural to presume that $D$ is not a derivation, i.e., 
$\sT'=\rho(D){\cdot}\sT\ne 0$. Let $\BU$ denote the linear span of $\sT, \sT'$ in $\Pi^1_2(\ca)$. If 
$\dim\BU=1$, then $D$ is an automorphism of $(\ca,\sT)$, up to a scalar. Therefore, it is always 
assumed that $\dim\BU=2$.
\\ \indent
{\sf (ii)} We assume below that $D$ is a {\it locally finite-dimensional\/} transformation, i.e., the linear span 
in $\ca$ of $\{D^nf\mid n=0,1,2,\dots\}$ is finite-dimensional for any $f\in\ca$. Then the exponential 
$\exp(D)\in GL(\ca)$ exists. This simplifying condition holds in many interesting examples of near-derivations (e.g. Example~\ref{ex:vinb} or Section~\ref{sect:near-Lie}).

Let $\cR$ be the natural representation of $GL(\ca)$ in $\Pi^1_2(\ca)$. If $g\in GL(\ca)$, then the 
bilinear operation corresponding to $\cR(g){\cdot}\sT$ is given by 
\[
    (x,y)\mapsto g{\cdot}\psi( g^{-1}{\cdot} x,g^{-1}{\cdot} y). 
\]
\begin{thm}                \label{thm:near1}
If $D$ is a locally finite-dimensional near-derivation of $(\ca,\sT)$, then there is a dense open subset 
$\Omega\subset\BU$ that consists of operations isomorphic to $\sT$. In particular, if 
$\tilde\sT\in \BU\setminus\Omega$, then $(\ca,\tilde\sT)$ is a degeneration 
of $(\ca,\sT)$ and $(\ca,\tilde\sT)$ satisfies all polynomial identities that $(\ca,\sT)$ satisfies.
\end{thm}
\begin{proof}   If $\rho(D)^2{\cdot}\sT=a\sT+b\sT'$, then 
$\rho(D)\vert_\BU={\small\begin{pmatrix} 0 &a \\1 & b\end{pmatrix}}$ \ and 
\[
    \cR(\exp(sD)){\cdot}\sT=\sT+s\rho(D){\cdot}\sT+\frac{s^2}{2!}\rho(D)^2{\cdot}\sT+\dots   \in \BU .
\]
Let $T_1$ denote the torus of scalar transformations in $GL(\BU)$. Since $\rho(D)\vert_\BU$ is not a 
scalar transformation, $K:=T_1{\cdot}\exp(sD)$ ($s\in\bbk$) is a 2-dimensional subgroup of $GL(\BU)$ 
such that the $K$-orbit of $\sT$, $K{\cdot}\sT$, is $2$-dimensional. Hence $\Omega=K{\cdot}\sT$ is dense in 
$\BU$. Clearly, all the tensors in $K{\cdot}\sT$ comprise the algebras isomorphic to $(\ca,\sT)$. Since 
$\tilde\sT$ lies in the closure of $K{\cdot}\sT$, the algebra $(\ca,\tilde\sT)$ is a degeneration of $(\ca,\sT)$ 
and all polynomial identities satisfied for $(\ca,\sT)$ are also satisfied for $(\ca,\tilde\sT)$.
\end{proof}

Our proofs of Theorems~\ref{thm:near2} and \ref{thm:near1} are adaptations of the proofs given by 
Vinberg~\cite{vi95} for Theorems~\ref{thm:v2} and \ref{thm:v1}. However, in this setting, some new 
phenomena may occur. For instance, it may happen that $\sT'\in K{\cdot}\sT$, i.e., the algebra $(\ca,\sT)$ 
is isomorphic to $(\ca,\sT')$. Then the essentially different operations on $\ca$ correspond to some other 
tensors in $\BU$. To see more clearly what is going on here, we first notice that

{\it If $D$ is a near-derivation and\/ $\BU=\lg \sT, \rho(D){\cdot}\sT\rg$ is the associated plane, then 
$\ap I+D$ is again a near-derivation, with the same $\BU$, for any $\ap\in\bbk$. Therefore, we may use 
the most suitable near-derivation of this form.}

\begin{thm}        \label{thm:new}
Let $D$ be a near-derivation of $(\ca,\sT)$ with the corresponding $2$-dimensional space 
$\BU=\lg \sT,\rho(D){\cdot}\sT\rg$. Then there is $D_1\in\lg I,D \rg\subset\gl(\ca)$ such that 
$\sT'_1=\rho(D_1)\sT$ is a  $D_1$-eigenvector and $\bd_1=\rho(D_1)\vert_\BU$ is either nilpotent or 
semisimple.
\end{thm}
\begin{proof}
We use below the notation of Theorem~\ref{thm:near1}.
For $\bd:=\rho(D)\vert_\BU={\small\begin{pmatrix} 0 &a \\1 & b \end{pmatrix}}\in\gl(\BU)$, there are the following possibilities. 

$({\bf a_1})$  Clearly, $\bd$ is nilpotent if and only if $a=b=0$. In this case, $D$ is a quasi-derivation in the sense of Vinberg. Then $\exp(t\bd)=I+t\bd\in GL(\BU)$ is unipotent and the group $K$ is non-abelian and solvable. More precisely, $K= \{{\small\begin{pmatrix} \ap  &0 \\ \beta & \ap \end{pmatrix}}\mid \ap\ne0\}\subset GL(\BU)$.

$({\bf a_2})$  If $a=0$ and $b\ne 0$, then $\bd$ is semisimple and $\sT'=\rho(D){\cdot}\sT$ is a 
$D$-eigenvector, with eigenvalue $b$. Then $\sT'\not\in K{\cdot}\sT$ and $\sT'$ certainly represents a new operation in $\ca$. Here $K$ is a two-dimensional torus in $GL(\BU)$, i.e.,
$K\simeq \{{\small\begin{pmatrix} \ap  & 0 \\0 & \beta \end{pmatrix}}\mid \ap\beta \ne0\}\subset GL(\BU)$.

\noindent 
\textbullet \ \ In these two cases, we do nothing and take $D_1=D$.

$({\bf a_3})$  If $a\ne 0$, then the following procedure  eliminates $a$.
Let $\{\lb_1, \lb_2\}$ be the eigenvalues of $\bd$. Hence $\lb_1+\lb_2=b$ and $\lb_1\lb_2=-a$. Then
$\bv_2=\left(\begin{smallmatrix} -\lb_1 \\ 1\end{smallmatrix}\right)=-\lb_1\sT+\sT'$ is a $\bd$-eigenvector with eigenvalue 
$\lb_2$. Let us replace $D$ with $D_1:=-\lb_1I+D$. Then the  eigenvalues of $D_1$ are 
$\{0, \lb_2-\lb_1\}$, \ $\rho(D_1){\cdot}\sT=-\lb_1\sT+\sT'$, and
\[
  \rho(D_1)^2{\cdot}\sT=D_1(-\lb_1\sT+\sT')=(\lb_2-\lb_1)(-\lb_1\sT+\sT') .
\]
Therefore, in the basis $(\sT,-\lb_1\sT+\sT')$ for $\BU$, we have $\bd_1=\rho(D_1)\vert_\BU=$
{\small $\begin{pmatrix}  0 & 0\\ 1& \lb_2-\lb_1 \end{pmatrix}$}.
Thus, for $D_1$, we get either the case $({\bf a_1})$ (if $\lb_1=\lb_2$), or the case $({\bf a_2})$ (if
$\lb_1\ne \lb_2$).
\end{proof}

Without loss of generality, we may (and will) assume below that $a=0$. Hence $D_1=D$ and the 
eigenvalues of $\bd_1=\bd$ are $\{0,b\}$. 

\begin{cl}[of the proof]   \label{cor:new}
For a near-derivation $D$ with $a=0$, there is a dichotomy:
\begin{itemize}
\item[\sf (1)] \ either $\rho(D)\vert_\BU$ is nilpotent and then $D$ is a quasi-derivation in the sense of 
Vinberg;
\item[\sf (2)] \  or $\rho(D)\vert_\BU$ is semisimple and then $\rho(D)^2{\cdot}\sT=b\rho(D){\cdot}\sT$ 
with $b\ne 0$.
\end{itemize}
In the f\un{irst} case, the group $K\subset GL(\BU)$ is non-abelian and solvable. Here
$\BU\setminus K{\cdot}\sT$ is a sole line and $\sT'=\rho(D){\cdot}\sT$ is the only degenerate operation 
in $\BU$ (up to a scalar). In the \un{second} case, $K$ is a $2$-dimensional torus in $GL(\BU)$ and 
$\BU\setminus K{\cdot}\sT$ is the union of two lines. This provides two degenerate operations in 
$\BU$ (up to a scalar). One of them is $\sT'$ and another one is $b\sT-\sT'$.
\end{cl}

\section{Near-derivations of Poisson algebras}
\label{sect:Near-Pois}
\noindent
As in Vinberg's Example~\ref{ex:vinb}, our main application (so far) concerns the case in which 
$\ca=\gS(\q)$ is the Poisson algebra of a Lie algebra $\q$ and $\wD$ is a near-derivation \wrt\ the 
Poisson bracket $\{\ ,\,\}$ in $\gS(\q)$.

\begin{thm}    \label{thm:near3}
Let $\wD$ be a near-derivation of\/ $(\gS(\q), \{\ ,\,\})$. Then
\begin{itemize}
\item[\sf (i)]  the Poisson brackets $\{\ ,\,\}$ and $\{\ ,\,\}'_{\wD}$ are compatible; 
\item[\sf (ii)]  the subalgebra of $(\gS(\q), {\cdot})$ generated by the elements $\wD^n(z)$ with $n\in \BN$ 
and $z\in\cz\gS(\q)$ is Poisson-commutative. 
\end{itemize}
\end{thm}
\begin{proof}
{\sf (i)} By Theorem~\ref{thm:near1}, the derived operation $\{\ ,\,\}'_{\wD}$ is again a Poisson bracket. 
Moreover, that proof shows that the dense open subset $K{\cdot}\sT\subset\BU$ consists of Poisson 
brackets isomorphic to $\{\ ,\,\}$. Therefore, one can find $\tilde\sT\in K{\cdot}\sT$ such that 
$\sT+\tilde\sT\in K{\cdot}\sT$, i.e., a sum of two Poisson brackets is again Poisson. This implies that all 
operations in $\BU$ are compatible Poisson brackets.

{\sf (ii)} This readily follows from Theorem~\ref{thm:near2} and the fact that the subalgebra of 
$(\gS(\q), {\cdot})$ generated by  pairwise Poisson-commuting elements is still 
Poisson-commutative.
\end{proof}
Thus, we have proved that 
\\ \indent
\textbullet \ \ a near-derivation $\wD$ of $(\gS(\q), \{\ ,\,\})$ provides the pencil of compatible 
Poisson brackets $\eus P=\{\ap\{\ ,\,\}+\beta\{\ ,\,\}'_{\wD} \mid \ap,\beta\in\bbk\}$  such that $\{\ ,\,\}$ is a 
generic element of $\eus P$; 
\\ \indent
\textbullet \ \ $\wD$ can be used for obtaining a \PC\ subalgebra of $\gS(\q)$.

Let $\cz_{\wD}$ be the \PC\ subalgebra of $\gS(\q)$ generated by the elements
$\wD^nz$, where $z\in\cz\gS(\q)$ and $n=0,1,2,\dots$ \ We provide below an alternate description of 
$\cz_{\wD}$.

The orbit $K{\cdot}\sT\subset\BU$ consists of all operations (Poisson brackets) that are isomorphic to 
$\sT$ (up to a scalar). Since the Poisson algebras corresponding to $\tilde\sT$ and $\ap\tilde\sT$
($\ap\in\bbk^\times$) have the same centre, we get all such centres by considering only the algebras
with $\tilde\sT\in \cR(\exp(s{\wD})){\cdot}\sT$, $s\in\bbk$. Let $\{\ ,\,\}_{(s)}$ denote the Poisson bracket in
$\gS(\q)$ corresponding to $\cR(\exp(s{\wD})){\cdot}\sT$, and let $\cz\gS(\q)_{(s)}$ be the corresponding centre. Then 
$\{\ ,\,\}_{(0)}=\{\ ,\,\}$ and $\cz\gS(\q)_{(0)}=\cz\gS(\q)$. Since all the brackets $\{\ ,\,\}_{(s)}$ are generic 
in $\eus P$, the subalgebra generated by $\cz\gS(\q)_{(s)}$ ($s\in\bbk$) is Poisson-commutative.
Let $\cz_\times$ denote this subalgebra.

\begin{thm}    \label{thm:sovpadenie}
Let $\wD$ be a near-derivation of\/ $(\gS(\q),\{\ ,\,\})$. Then $\cz_{\wD}=\cz_\times$.
\end{thm}
\begin{proof}   
To simplify notation, set $e^{s{\wD}}:=\cR(\exp(s{\wD}))$. Then we have 
\[
      \{F,H\}_{(s)}=e^{s{\wD}}\{e^{-s{\wD}}(F), e^{-s{\wD}}(H)\}
\] 
for $F,H\in\gS(\q)$. Hence $F\in \cz\gS(\q)_{(s)}$ if and only if $e^{-s{\wD}}{\cdot}F\in \cz\gS(\q)$. 
Therefore, $\cz_\times$ is the $\bbk$-algebra generated by 
$\{e^{s{\wD}}(F) \mid F\in\cz\gS(\q), s\in\bbk\}$. Here
\[
     e^{s{\wD}}(F)=F+s {\wD}(F)+\frac{s^2}{2!}{\wD}^2(F)+\frac{s^3}{3!}{\wD}^3(F)+\dots
\]
As $\wD$ is locally finite-dimensional, the linear span of $\{e^{s{\wD}}{\cdot}F\mid s\in\bbk\}$ is 
finite-dimensional and coincides with the linear span of $\{{\wD}^n(F)\mid n=1,2,\dots\}$ for each 
$F\in \cz\gS(\q)$ (use the Vandermonde determinant).
\end{proof}

\begin{rmk}
It is convenient to have a near-derivation $\wD$ of $(\gS(\q),\{\ ,\,\})$ that is also a derivation of 
$(\gS(\q),{\cdot})$. This may simplify our description of $\cz_{\wD}$ as follows. Suppose that $\cz\gS(\q)$ 
is finitely generated and $F_1,\dots,F_l$ is a generating set.  Then the linear span of 
$\{\wD^n(F_j)\mid j=1,\dots,l \ \& \ n=0,1,2,\dots\}$ is finite-dimensional, and it generates $\cz_{\wD}$.
\end{rmk}

In order to obtain such $\wD$, 
one may begin with $D\in {\gl}(\q)$ and extend it to the linear transformation $\widehat D$ of $\gS(\q)$ 
using the Leibniz rule, i.e., we set $\widehat D(x{\cdot}y)=(Dx){\cdot}y+x{\cdot}(Dy)$ for $x,y\in\q$, and so 
on. Since $\widehat D$ preserves each space $\gS^m\q$, it is locally finite-dimensional. Hence previous 
results of this section apply to the near-derivations $\widehat D$ of this form. The key problem is to find $D$ such that $\widehat D$ is a near-derivation of $(\gS(\q), \{\ ,\,\})$.  

In this setting, we will work with the derived operation of the Lie bracket $[\ ,\,]$ with respect to $D$. 
According to the general definition, it is given by
\beq         \label{eq:derived-Lie}
     [x,y]'_{D}=D([x,y])-[Dx,y]-[x,Dy] .
\eeq
Note that Eq.~\eqref{eq:derived-Lie} represents, up to sign, the usual coboundary operator
$\delta: \gC^1(\q,\q)\to \gC^2(\q,\q)$, i.e., $[\ ,\,]'_D=-\delta(D)$. Given $[x,y]'_{D}$,  the derived operation 
$\{\ ,\,\}'_{\widehat D}$ in $\gS(\q)$ is obtained from $[\ ,\,]'_{D}$ via the use 
of the Leibniz rule. Combining this with results of Section~\ref{sect:near}, we obtain the following 
conclusions.
\begin{lm}                 \label{lm:near-Lie}
For $D\in\gl(\q)$ and its extension $\widehat D:\gS(\q)\to \gS(\q)$, we have
 \begin{itemize}
\item[\sf (1)] \ $[\ ,\,]'_{D}$ is a Lie bracket \emph{ if and only if} $\{\ ,\,\}'_{\widehat D}$ is a Poisson 
bracket;
\item[\sf (2)] \  $D$ is a near-derivation (resp. quasi-derivation) of $(\q,[\ ,\,])$ \emph{ if and only if} 
$\widehat D$ is a near-derivation (resp. quasi-derivation) of\/ $(\gS(\q), \{\ ,\,\})$;
\item[\sf (3)] \  If $D$ is a near-derivation of $(\q,[\ ,\,])$, then the Lie brackets $[\ ,\,]$ and $[\ ,\,]'_{D}$
are compatible.
\end{itemize}
\end{lm}
Having this lemma, we usually forget about $\widehat D$ and deal mainly with $D\in{\gl}(\q)$. In
particular, if $D$ is a near-derivation of $\q$, then we write $\cz_D$ in place of $\cz_{\widehat D}$ for the
corresponding \PC\ subalgebra of $\gS(\q)$. Also, now $\BU\subset \Pi_2^1(\q)$ is the vector space 
generated by the tensors corresponding to the Lie brackets $[\ ,\,]$ and $[\ ,\,]'_D$. As before, we will 
always assume that $\dim\BU=2$.

\section{Near-derivations of Lie algebras: examples}
\label{sect:near-Lie}
\noindent
In this section, we consider several (old and new) examples of compatible Poisson brackets in $\gS(\q)$
that are related to near-derivations of $\q$.

\subsection{Periodic gradings and near-derivations}        
\label{subs:period}
Let $\vth$ be an automorphism of $\q$ of order $n$ and $n\ge 2$. Fix a primitive root of unity 
$\zeta=\sqrt[n]1$ and consider the $\vth$-eigenspaces $\q_i=\{x\in\q\mid \vth(x)=\zeta^ix\}$ with 
$i=0,1,\dots,n-1$. Then the decompossition 
\beq   \label{eq:periodic} 
          \q=\q_0\oplus\q_1\oplus\ldots\oplus\q_{n-1}  
\eeq
is a $\BZ_n$-{\it grading} (or just a {\it periodic grading\/}) of $\q$. Here 
$[\q_i,\q_j]\subset \begin{cases}\q_{i+j}, & i+j\le n{-}1 \\ \q_{i+j-n}, & i+j\ge n
\end{cases}$. \\
For $x=\sum_{i=0}^{n-1}x_i\in\q$, we set $Dx=\sum_{i=0}^{n-1}ix_i$.

\begin{prop}   \label{lm:D-g}
The operator $D\in {\gl}(\q)$ is a near-derivation of the $\BZ_n$-graded Lie algebra $\q$. 
\end{prop}
\begin{proof}
Let us compute the second derived operation of the Lie bracket $[\ ,\,]$ with respect to $D$.  While 
computing $[x,y]'_D$, we may assume that $x=x_i\in\q_i$ and $y=y_j\in\q_j$. Then
\[
     [x_i,y_j]'_D:=D([x_i,y_j])-[Dx_i,y_j]-[x_i,Dy_j] = D([x_i,y_j]) - (i+j)[x_i,y_j] .
\]
Hence 
$\displaystyle
       [x_i,y_j]'_D{=}\begin{cases}  0, &  i+j{<}n ; \\ -n[x_i,y_j], & i+j{\ge} n . \end{cases}
$
 \ A similar computation shows that
$  [x_i,y_j]''_D=-n[x_i,y_j]'_D$.
Thus, the matrix of $\rho(D)\vert_\BU$ with respect to $[\ ,\,]$ and $[\ ,\,]'_D$ equals 
{\small $\begin{pmatrix} 0 & 0 \\ 1 & -n
\end{pmatrix}$}.
\end{proof}
\begin{rmk}      \label{rem:3.3}
Having replaced $D$ with $\eus D:=\frac{1}{n}D$, one obtains simpler formulae, i.e.,
\[
     [x_i,y_j]''_\eus D=-[x_i,y_j]'_\eus D=\begin{cases}  0, &  i+j<n ; \\ [x_i,y_j], & i+j\ge n . \end{cases}
\ \text{ and }  \ \rho(\eus D)\vert_\BU=\begin{pmatrix} 0 & 0 \\ 1 & -1
\end{pmatrix}.
\]
But our choice of $D$ is related to a polynomial transformation of $\q$, see below.  
\end{rmk}

By Theorem~\ref{thm:near3} and Lemma~\ref{lm:near-Lie}, $\eus P=\{\ap[\ ,\,]{+}\beta[\ ,\,]'_D\mid \ap,\beta\in\bbk\}$ is a pencil of 
compatible Lie brackets. The associated \PC\ subalgebra of 
$\gS(\q)$ has already been studied in~\cite{period1}. Let us recall some relevant results and relate 
them to the near-derivation $D$.

Let $\vp: \bbk^\times \to GL(\q)$ be the 1-parameter group such that $\vp_t(x)=\sum_{i=0}^{n-1}t^i x_i$,
where $t\in\bbk^\times$ and we write $\vp_t$ in place of $\vp(t)$. The mappings $\q\times\q\to \q$ such that 
\[
    (x,y) \mapsto [x,y]_{\lg t\rg}:= \vp_t^{-1}[\vp_t(x), \vp_t(y)]  \quad (t\in\bbk^\times)
\]
provide a family of Lie brackets isomorphic to the initial one. Here $[x,y]_{\lg 1\rg}=[x,y]$ and the Poisson 
bracket in $\gS(\q)$ corresponding to $[\cdot ,\cdot]_{\lg t\rg}$ is denoted by $\{\cdot ,\cdot\}_{\lg t\rg}$.
An important feature of $\BZ_n$-gradings is that here $[x,y]_{\lg t\rg}$ is a polynomial in $t$ of degree 
$n$ with only two terms~\cite[Prop.\,2.1]{period1}; namely, 
\[
      [x,y]_{\lg t\rg}=[x,y]_{(0)}+t^n [x,y]_{(\infty)}.
\]
Hence $[x,y]_{(0)}=\lim_{t\to 0} \vp_t^{-1}[\vp_t(x), \vp_t(y)]$. By~\cite{period1}, $[\ ,\,]_{(0)}$ and 
$[\ ,\,]_{(\infty)}$ are compatible Lie brackets. From the present point of view, this can be explained as follows.
Given $\vp$, the near-derivation $D$ is obtained as $Dx=(\vp_t(x))'\vert_{t=1}$. Then a straightforward 
computation shows that $[x,y]'_D=-n{\cdot} [x,y]_{(\infty)}$. Hence 
\[
   [\ ,\,]_{(0)}=[\ ,\,]-[\ ,\,]_{(\infty)}
\]
is also a Lie bracket. The pencil $\eus P$ has a natural basis $([\ ,\,]_{(0)}, [\ ,\,]_{(\infty)})$ that consists
of the $D$-eigenvectors (cf. Corollary~\ref{cor:new}) and the corresponding degenerate Lie algebras, 
$\q_{(0)}$ and $\q_{(\infty)}$, have also been studied in~\cite{period1}. They are $\BN$-graded. More 
precisely, the component of grade $i$ in $\q_{(0)}$, $\q_{(0)}(i)$, is equal to $\q_i$ ($i=0,1,\dots,n-1$) and $[x_i,y_j]_{(0)}=  \begin{cases}
 [x_i,y_j], & i+j < n; \\ 0, &  i+j\ge n.
\end{cases}$ \quad
To distinguish $\q_{(0)}$ from the direct sum in~\eqref{eq:periodic}, this $\BN$-graded contraction of
$\q$ is denoted as $\q_{(0)}=\q_0\ltimes\q_1\ltimes\ldots\ltimes\q_{n-1}$.

The Lie bracket $[\ ,\,]_{(\infty)}=-[\ ,\,]'_\eus D$ is given in Remark~\ref{rem:3.3}. This shows that the 
$\BN$-graded structure of $\q_{(\infty)}$ is such that $\q_{(\infty)}(j)=\q_{n-j}$ \ ($j=1,2,\dots,n$). Hence 
$\q_{(\infty)}$ is positively graded and therefore nilpotent. 

Using $D$ and $\cz\gS(\q)$, one obtains a \PC\ subalgebra $\cz_{\wD}=\cz_D\subset \gS(\q)$, see 
Theorem~\ref{thm:near3}{\sf (ii)}. By Theorem~\ref{thm:sovpadenie}, one has $\cz_D=\cz_\times$. In 
the setting of periodic gradings this algebra is also described as follows. For each $m\in\BN$, we
extend $\vp$ to the homomorphism $\bbk^\times\to GL(\gS^m\q)$. 
Then $\vp_t(f)=\sum_{j\ge 0} t^jf_j$ is a finite sum for any $f\in\gS(\q)$.
If $f$ is homogeneous, then 
the nonzero elements $\{f_j\}$ are called the {\it bi-homogeneous components\/} of $f$ (\wrt\ $\vp$).
Using the Vandermonde determinant and relations between $D$, $\widehat D$, and $\vp$, we see that, for each $f\in\gS^m\q$, the linear span of $\{\widehat D^n (f)\}_{n\ge 0}$ coincides with the 
linear span of all bi-homogeneous components of $f$. Hence $\cz_\times$ is
the algebra generated by the bi-homogeneous components of all $f\in\cz\gS(\q)$. 
This connection between $\cz_\times$ and the bi-homogeneous components of the elements of
$\cz\gS(\q)$ has been established in~\cite[Sect.\,4]{period1}. 

\begin{rmk}     \label{rem:q-versus-g}
The main results on the \PC\ subalgebra associated with $\vth$ and the pencil 
$\eus P$ are obtained in~\cite{period1} for the {\bf reductive} Lie algebras $\q$. But the preliminary
results related to $\cz_\times$ and the bi-homogeneous components of $f\in\gZ\gS(\q)$
actually hold for arbitrary Lie algebras.
\end{rmk}

\subsection{Quasi-gradings and near-derivations}      \label{subs:ext-period}
A similar construction works if $(\q,\vth,n)$ are as in Section~\ref{subs:period} and we consider the 
Lie algebra
\[
   \tfq=\q\dotplus\q_0=(\q_0\oplus\q_1\oplus\dots\oplus\q_{n-1})\dotplus\q_0 . 
\]
To distinguish the two copies of $\q_0$ in $\tfq$, the direct summand (ideal of $\tfq$) will be denoted 
below by $\q_{[0]}$. Consider the vector space decomposition 
\beq   \label{eq:tq-decomp}
    \tfq=\tilde\q_0\oplus\tilde\q_1\oplus\dots\oplus\tilde\q_{n-1}\oplus\tilde\q_n 
\eeq 
such that  $\tilde\q_0=\Delta\q_0\subset\q_0\dotplus\q_{[0]}$, \ $\tilde\q_i=\q_i$ ($i=1,2,\dots,n{-}1$), \ 
and $\tilde\q_n=\q_0$.  Although it is not a periodic grading, it behaves, to a some extent, like a 
$\BZ_{n+1}$-grading. Namely, $[\tilde\q_i,\tilde\q_j]\subset \tilde\q_{i+j}$ if $i+j\le n$. Note that the 
extreme terms, $\tfq_0$ and $\tfq_n$, are subalgebras of $\tfq$ and $\tfq_0\simeq\tfq_n$. 
Decomposition~\eqref{eq:tq-decomp} appears first in~\cite[Lemma\,2.2]{p09}, where it is called a 
{\it quasi-grading} of $\tfq$. It is shown in~\cite{p09} that $\tfq$ admits a contraction to the $\BN$-graded 
Lie algebra 
\[ 
   \tfq_{(0)}=\tfq_0\ltimes\tfq_1\ltimes\dots\ltimes\tfq_n=
   \Delta\q_0\ltimes\q_1\ltimes\dots\ltimes\q_{n-1}\ltimes\q_0 ,
\]
the component of grade $i$ in $\tfq_{(0)}$ being $\tfq_i$. If $\tx_i\in\tfq_i$ and $\ty_j\in\tfq_j$, then
the Lie bracket in $\tfq_{(0)}$ is given by 
$[\tx_i,\ty_j]_{(0)}=\begin{cases}  [\tx_i,\ty_j], &  i+j\le n, \\ 0, & i+j\ge n+1 .
\end{cases}$. This means that $\tfq_0$ remains a subalgebra in 
$\tfq_{(0)}$. The Lie algebra $\tfq_{(0)}$ has further been studied in \cite{ukr}. In particular, it was shown therein that if $\q=\g$ is reductive, then $\ind\tfg_{(0)}=\ind\tfg=\rk\g+\rk\g_0$.

Our recent observation is that the sum in~\eqref{eq:tq-decomp} always provides a near-derivation 
of $\tfq$ and thereby compatible Lie  brackets in $\tfq$.
For $\tx=\sum_{i=0}^n \tx_i\in \tilde\q$, set $D\tx=\sum_{i=0}^n i\tx_i$.
\begin{lm}            \label{lm:D-tilde-g}
The operator $D\in \gl(\tilde\q)$ is a near-derivation of the Lie algebra $\tilde\q$.
\end{lm}
\begin{proof}
Arguing as in Proposition~\ref{lm:D-g}, we obtain 
$[\tx_i,\tilde y_j]'_D=\begin{cases}  0, &  i+j<n+1, \\ -n[\tx_i,\tilde y_j], & i+j\ge n+1 .\end{cases}$   
\\ 
and $[\tx_i,\tilde y_j]''_D=-n[\tx_i,\tilde y_j]'_D$. 
\end{proof}

Consider the 1-parameter group $\vp: \bbk^\times \to GL(\tfq)$ such that 
$\vp_t(\tx)=\sum_{i=0}^{n}t^i \tx_i$. A direct computation shows that the deformed Lie bracket is again a
polynomial in $t$ of degree $n$, with only two terms. That is, one obtains
\[
      \vp_t^{-1}[\vp_t(\tx), \vp_t(\ty)] =[\tx,\ty]_{(0)}+t^n [\tx,\ty]_{(\infty)},
\]
where again $\lim_{t\to 0}\vp_t^{-1}[\vp_t(\tx), \vp_t(\ty)]=[\tx,\ty]_{(0)}$ and
$[\tx,\ty]_{(\infty)}{=}-\frac{1}{n} [\tx,\ty]'_D$.
Although this looks rather similar to the situation of Section~\ref{subs:period}, here the degenerate Lie 
algebras $\tfq_{(0)}$ and $\tfq_{(\infty)}$ are quite different from $\q_{(0)}$ and $\q_{(\infty)}$. 
\\ \indent
\textbullet \quad Recall that a  Lie algebra $\q$ is {\it quadratic}, if there is a $\q$-invariant symmetric
bilinear non-degenerate form $\eus B$ on $\q$. It is shown in \cite[Corollary\,3.2]{p09} that if $(\q, \eus B)$ 
is quadratic and $\eus B$ is $\vth$-invariant, then $\tfq_{(0)}$ is also quadratic. This holds, in particular, if
$\q$ is semisimple.
\\ \indent
\textbullet \quad 
The explicit formulae for the Lie bracket $[\ ,\,]_{(\infty)}=[\ ,\,]-[\ ,\,]_{(0)}$
show that $\tfq_{(\infty)}$ is also $\BN$-graded. Namely, 
$\tfq_{(\infty)}(j)=\tfq_{n-j}$ for $j=0,1,\dots,n-1$ and $\tfq_{(\infty)}(n)=\tfq_0$. Therefore,
$\tfq_{(\infty)}$ is a direct sum of
the abelianisation of $\tfq_0$ and the periodic contraction of $\q$ corresponding to $\vth^{-1}$, i.e.,
\[
      \tfq_{(\infty)}\simeq (\tfq_0)^{\sf ab}\dotplus (\tfq_n\ltimes\tfq_{n-1}\ltimes\ldots\ltimes\tfq_1)
      = (\Delta\q_0)^{\sf ab}\dotplus (\q_0\ltimes\q_{n-1}\ltimes\ldots\ltimes\q_1).
\]
The study of Poisson-commutative subalgebras related to quasi-gradings and the pencil 
$\eus P=\lg [\ ,\,]_{(0)}, [\ ,\,]_{(\infty)}\rg$ requires the 
knowledge of the indices of Lie algebras  $\tfq_{(0)}$ and  $\tfq_{(\infty)}$ and certain properties of the
Poisson centre $\gZ\gS(\tfq_{(\infty)})$. This will be the subject of a forthcoming paper.

\subsection{Splittings and near-derivations}   
\label{subs:2-split}
Another good case occurs if $\q=\h\oplus\rr$ is a {\it splitting\/} of $\q$, i.e., it is a sum of two subalgebras 
(as vector space). Let ${\sf pr}_\h: \q\to \h$ and ${\sf pr}_\rr: \q\to \rr$ be the associated projections. For 
$x_i=h_i+r_i$ ($h_i\in\h$, $r_i\in\rr$), consider the operators $D_1(x_i)=h_i$ and $D_2(x_i)=r_i$.
Then a straightforward computation shows that
\[
    [x_1,x_2]''_{D_1}=-[x_1,x_2]'_{D_1}=[h_1,h_2]+{\sf pr}_\rr([h_1,r_2]+[r_1,h_2])
\]
and likewise for $D_2$. Hence both $D_1$ and $D_2$ are near-derivations of $\q$ and
\[
      -[x_1,x_2]'_{D_1}-[x_1,x_2]'_{D_2}=[x_1,x_2].
\]
The Lie bracket \ $-[\ ,\,]'_{D_1}$ corresponds to the semi-direct product $\h\ltimes\rr^{\sf ab}$, where the 
abelian ideal $\rr^{\sf ab}$ is regarded as $\h$-module $\q/\h$. Likewise, the Lie bracket \ $-[\ ,\,]'_{D_2}$ 
corresponds to the semi-direct product $\rr\ltimes\h^{\sf ab}$. The pencil 
$\ap [\ ,\,]'_{D_1}+\beta [\ ,\,]'_{D_2}$ of compatible Poisson brackets and associated 
\PC\ subalgebras of $\gS(\q)$ have been studied in~\cite{bn}.

Let $\vp: \bbk^\times \to GL(\q)$ be the 1-parameter group such that $\vp_t(x)=h+tr$. Then
\[
    [x_1,x_2]_{\lg t\rg}:= -[x_1,x_2]'_{D_1}-t [x_1,x_2]'_{D_2}.
\]
Here $\lim_{t\to 0}[\ ,\,]_{\lg t\rg}=-[\ ,\,]'_{D_1}$. Therefore $-[\ ,\,]'_{D_1}=[\ ,\,]_{(0)}$ and 
$-[\ ,\,]'_{D_2}=[\ ,\,]_{(\infty)}$ in the notation of~\cite{bn}. Since $D_1$ and $D_2$ play similar roles,
we concentrate below on $D:=D_1$.

The sum $\q=\h\oplus\rr$ provides the bi-homogeneous decomposition
\[
      \gS^p\q=\bigoplus_{m=0}^p(\gS^m\h\otimes\gS^{p-m}\rr) .
\]
The eigenvalue of the extended operator $\widehat D$ on $\gS^m\h\otimes\gS^{p-m}\rr$ equals $m$. 
Therefore, for a homogeneous $f\in\gS(\q)$, the linear span of $\{\widehat D^n (f)\}_{n\ge 0}$ coincides 
with the linear span of all bi-homogeneous components of $f$. Hence the \PC\ 
algebra $\cz_D\subset\gS(\q)$  coincides with the algebra generated by the bi-homogeneous components of all $f\in\cz\gS(\q)$.
By~\cite[Sect.\,3]{bn} or Theorem~\ref{thm:sovpadenie}, the latter is also the algebra generated by the Poisson centres of the Poisson brackets $\{\ ,\,\}_{(0)}+t\{\ ,\,\}_{(\infty)}$ with $t\in\bbk^\times$.
For the reductive $\q$, the \PC\ subalgebra of $\gS(\q)$ corresponding to
$\eus P=\lg [\ ,\,]_{(0)}, [\ ,\,]_{(\infty)}\rg$ has been studied in \cite{bn}.

{\it\bfseries Summary.} 
For all three cases in Sections~\ref{subs:period}-\ref{subs:2-split}, $D\in\gl(\q)$ is a semisimple 
operator with integral eigenvalues. Then the associated $1$-parameter group $\vp_t$ has the 
property that the deformed Lie bracket $(x,y)\mapsto\vp_t^{-1}[\vp_t(x),\vp_t(y)]$ appears to be a 
polynomial in $t$ with only {\bf two} terms. Since $D$ is a semisimple near-derivation, the plane $\BU$ 
contains two degenerate Lie brackets, $[\ ,\,]_{(0)}$ and $[\ ,\,]_{(\infty)}$,  (cf. Corollary~\ref{cor:new}).
\\ \indent
In Example~\ref{ex:vinb}, the quasi-derivation $D_\gamma$ of $\{\gS(\q), \{\ ,\,\})$ is not of the form 
$\widehat D$ for some $D\in\gl(\q)$, since it does not preserve the $\BN$-grading in $\gS(\q)$. This 
is the only known example of a near-derivation of $\gS(\q)$ that does not arise from a linear 
transformation of $\q$. At the moment, there is still no examples of quasi-derivations (=\,nilpotent 
near-derivations) of Lie algebras. They will appear in Section~\ref{sect:quasi-Z}, where
we discuss general Lie-algebraic constructions of quasi-derivations.

\section{Special $1$-parameter groups and near-derivations}
\label{sect:2term-deformations}

As usual, $\q$ is a finite-dimensional Lie algebra. Let $\vp:\bbk^\times\to GL(\q)$ be a {polynomial} 
homomorphism of algebraic groups. Then $\q=\bigoplus_{i= 0}^p\q_i$ is a direct sum of 
$\vp$-eigenspaces, where $\vp_t(x_j)=t^jx_j$ for $x_j\in\q_j$. We assume that $\q_p\ne 0$, hence 
$\deg_t(\vp_t(x))=p$. (It is not assumed that $\vp$ is related to a grading of $\q$.) Recall that
$[x,y]_{\lg t\rg}=\vp_t^{-1}[\vp_t(x),\vp_t(y)]$ ($t\in\bbk^\times$) is the {\it deformed\/} Lie bracket, where 
$[\ ,\,]_{\lg 1\rg}$ is the initial Lie bracket in $\q$. 
Whenever we wish to stress the dependance on $\vp$, it will be called the $\vp$-{\it deformed\/} Lie 
bracket.
\\ \indent
The deformed bracket is not necessarily a polynomial in $t$. In general, $[x,y]_{\lg t\rg}$ is only a Laurent 
polynomial in $t$. A proof of the following criterion is easy and left to the reader.

\begin{lm}     \label{lm:quasi-grad}
For a polynomial homomorphism $\vp$, the $\vp$-deformed bracket  $[x,y]_{\lg t\rg}$ is a polynomial in 
$t$ if and only if 
\beq       \label{eq:quasi-grad}
      [\q_i,\q_j]\subset \bigl(\bigoplus_{k\le i+j}\q_k\bigr)
\eeq
for all $i,j$. In particular, if this is the case, then $\q_0$ is a Lie subalgebra and 
$[\q_0,\q_i]\subset\bigoplus_{k\le i}\q_k$.
\end{lm}

It is assumed below that condition~\eqref{eq:quasi-grad} is satisfied. Then  
\[ 
     \textstyle  [x,y]_{\lg t\rg}=\sum_{i=0}^m  [x,y]_{(i)}t^i, 
\] 
where $[\ ,\,]_{(i)}$ are bilinear operations on $\q$. If $[\ ,\,]_{(m)}\ne 0$, then we set
$\deg_t([x,y]_{\lg t\rg}):=m$.

\begin{lm}   \label{lm:5.2}
If\/ $[x,y]_{\lg t\rg}$ is a polynomial of degree $m$, then 
$[\ ,\,]_{(0)}$ and $[\ ,\,]_{(m)}$ are again Lie brackets on the vector space $\q$.
\end{lm}
\begin{proof}
Here $[\ ,\,]_{\lg t\rg}$ is a Lie bracket for $t\in\bbk^\times$ and
$[[x,y]_{\lg t\rg},z]_{\lg t\rg}=\sum_{i=0}^m\sum_{j=0}^m t^{i+j} [[x,y]_{(i)},z]_{(j)}$ is a polynomial of
degree $2m$. Here the coefficient of $t^0$ (resp. $t^{2m}$) is equal to $[[x,y]_{(0)},z]_{(0)}$ (resp. 
$[[x,y]_{(m)},z]_{(m)}$). This implies everything.
\end{proof}

However, the Lie brackets $[\ ,\,]_{(0)}$ and $[\ ,\,]_{(m)}$ are not necessarily compatible. We consider 
below a sufficient condition for compatibility.

Set $D=\vp'\vert_{t=1}$, where `prime' denotes the derivative by $t$. Then $D\in{\gl}(\q)$ and 
$Dx=\vp_t(x)'\vert_{t=1}$. In other words, $Dx=\sum_{i=0}^p ix_i$ for $x=\sum_i x_i$. The following 
standard result provides a relation between the derivatives at $t=1$ of $\vp$ and the $\vp$-deformed bracket.

\begin{prop}
We have  \  $([x,y]_{\lg t\rg})'\vert_{t=1}=[Dx,y]+[x,Dy]-D([x,y])=-[x,y]'_D$. 
\end{prop}

\begin{df}    \label{def:special}
A polynomial homomorphism $\vp:\bbk^\times \to GL(\q)$ is said to be {\it special}, if the $\vp$-deformed 
Lie bracket is a polynomial with only two terms, i.e., \\
\centerline{$[x,y]_{\lg t\rg}= [x,y]_{(0)}+t^m [x,y]_{(m)}$.} \\[.6ex]
In this case, we may also say that $\vp$ is $m$-{\it special}.
\end{df}

For a special $\vp$, one has more  properties of the $\vp$-eigenspaces in $\q$.

\begin{lm}      \label{lm:more-prop-special}
If $\vp$ is $m$-special, then $[\q_i,\q_j]\subset \q_{i+j}\oplus \q_{i+j-m}$. Consequently, if
$i+j<m$, then $[\q_i,\q_j]\subset \q_{i+j}$; and if $i+j>p$, then $[\q_i,\q_j]\subset \q_{i+j-m}$.
\end{lm}

\begin{thm}       \label{thm:2-terms}
Suppose that $\vp$ is $m$-special.  Then  \\ \indent 
(1) \ $[\ ,\,]_{(0)}$ and  $[\ ,\,]_{(m)}$ are compatible Lie brackets; \\ \indent
(2) \ $D$ is a semisimple near-derivation of\/ $\q$ and $[\ ,\,]'_D=-m[\ ,\,]_{(m)}$.
\end{thm}
\begin{proof}  
Part (1) follows from Lemma~\ref{lm:5.2}. It is also a formal consequence of part (2) and Theorem~\ref{thm:near3}.

To prove (2), we honestly compute the second derived operation of $[\ ,\,]$ \wrt\ $D$. Recall that
$[x_i,y_j]'_D=D([x_i,y_j])-(i+j)[x_i,y_j]$ for $x_i\in\q_i,\, y_j\in\q_j$. Hence $[x_i,y_j]'_D=0$ whenever 
$i+j<m$. In general, we have $[x_i,y_j]=a_{i+j-m}+a_{i+j}$, where $a_k\in\q_k$. Therefore, 
\[ 
   [x_i,y_j]'_D=(i+j-m)a_{i+j-m}+(i+j)a_{i+j}-(i+j)[x_i,y_j] =-m{\cdot}  a_{i+j-m}.
\] 
Then 
\begin{multline*}
   [x_i,y_j]''_D=D([x_i,y_j]'_D) - [Dx_i,y_j]'_D - [x_i,Dy_j]'_D  = \\
   mi{\cdot} a_{i+j-m}+mj{\cdot} a_{i+j-m}-m(i+j-m)a_{i+j-m}=m^2a_{i+j-m} .
\end{multline*}
Thus, $[\ ,\,]''_D= -m [\ ,\,]'_D$, and $D$ is a semisimple near-derivation. 

It remains to notice that $m[\ ,\,]_{(m)}=-[\ ,\,]'_D$. Indeed, using the previous notation, we have
$[x_i,y_j]_{\lg t\rg}=a_{i+j}+t^m a_{i+j-m}$. Hence
$[x_i,y_j]_{(0)}=a_{i+j}$ and $[x_i,y_j]_{(m)}=a_{i+j-m}$.
\end{proof}

For the $m$-special homomorphisms $\vp$, the Lie bracket $[\ ,\,]_{(m)}$ can also be denoted by 
$[\ ,\,]_{(\infty)}$, which agrees with notation in Section~\ref{sect:near-Lie}. For, in this case we
have the $1$-parameter family of compatible Lie brackets $[\ ,\,]_{\lg t\rg}=[\ ,\,]_{(0)}+t[\ ,\,]_{(m)}$, $t\in\BP=\bbk\cup\{\infty\}$,
and the value $t=\infty$ gives rise to $[\ ,\,]_{(m)}$.

In the examples of Sections~\ref{subs:period}--\ref{subs:2-split}, one has a strong relation between 
$m=\deg_t ([x,y]_{\lg t\rg})$ and $p=\deg_t\vp$. Namely, $p\in\{m{-}1,m\}$. In general, this seems to be 
not the case. However, there is an easy partial observation.
\begin{prop}          \label{prop:easy}
Let $\vp$ be an $m$-special homomorphism. \\  \indent
(1) \ If\/ $[\q_p,\q_p]\ne 0$, then $p\le m$; \\ \indent
(2) \ if\/ $\q_p$ is a non-abelian subalgebra, then $p=m$. 
\end{prop}
\begin{proof}
By Lemma~\ref{lm:more-prop-special}, we have $0\ne [\q_p,\q_p]\subset \q_{2p-m}$. Hence
$2p-m\le m$. Moreover, if $\q_p$ is a subalgebra, then $2p-m=p$.
\end{proof}

On the other hand, it follows from Lemma~\ref{lm:more-prop-special} that if $p<i+j<m$, then
$[\q_i,\q_j]=0$.  This shows that if there are ``sufficiently many'' pairs $(i,j)$ such that  
$[\q_i,\q_j]\ne 0$, then $p\in\{m{-}1,m\}$.

\begin{rmk}
So far, we deal with an arbitrary Lie algebra $\q$. However, it is conceivably possible that,
for semisimple Lie algebras, one can obtain stronger results on special homomorphisms $\vp$
and properties of the $\vp$-eigenspaces.
\end{rmk}

\section{Quasi-derivations, $\BZ$-gradings, and nilpotent elements}    
\label{sect:quasi-Z}
\noindent
As above, $\q$ is an arbitrary algebraic Lie algebra. If $D\in\gl(\q)$ and $D^2=0$, then it
follows from Eq.~\eqref{eq:2-proizv} that $[x,y]''_D=2[Dx,Dy]-2D([Dx,y]+[x,Dy])$.
This readily implies that $D$ is a quasi-derivation of $\q$ whenever the following conditions are satisfied:
\begin{itemize}
\item[\sf (i)] \  $D^2=0$;
\item[\sf (ii)] \  $[\Ima D,\Ima D]=0$, i.e., $\Ima D$ is an abelian subalgebra of $\q$;
\item[\sf (iii)] \  $[\Ima D,\q]\subset \Ker D$.
\end{itemize}
Explicit examples can be produced via $\BZ$-gradings. Let $\q=\bigoplus_{i\in\BZ}\q(i)$ be a 
$\BZ$-grading of the Lie algebra $\q$, i.e., $[\q(i),\q(j)]\subset \q(i{+}j)$. For simplicity, assume that 
$\dim\q(j)=\dim\q(-j)$ for all $j$. This is always the case, if $\q$
is quadratic (e.g. reductive). Set $\q({\ge}j)=\bigoplus_{i\ge j}\q(i)$. If $m=\max\{j\mid \q(j)\ne 0\}$,
then one may take  $D$ such that \\[.4ex]
\centerline{ 
$\Ima D\subset \q({\ge} [\frac{m}{2}]{+}1)$ \ and \ $\Ker D\supset \q({\ge}{-}[\frac{m}{2}])$.}
\\[.4ex]
Then conditions {\sf (i)}--{\sf (iii)} above are satisfied and $D$ is a quasi-derivation of $\q$. Note that 
\\[.5ex]
\centerline{$\dim \q({\ge} [\frac{m}{2}]{+}1)+\dim \q({\ge}{-}[\frac{m}{2}])=\dim\q$.}
\begin{ex}  
\label{ex:short-Z}
{\bf 1.} The simplest case occurs if $\q=\g$ is simple and $\g=\g(-1)\oplus\g(0)\oplus\g(1)$. Such "short" 
$\BZ$-gradings exist, if $\g\not\in\{\GR{G}{2},\GR{F}{4}, \GR{E}{8}\}$. 
Then one can find $D$ such that  $\Ima D=\g(1)$, and $\Ker D=\g(0)\oplus\g(1)$.
\\ \indent
{\bf 2.} By a classical result of Vinberg, the group $G(0)$ has finitely many orbits in $\g(1)$.
If the open $G(0)$-orbit $\co_1\subset\g(1)$ is {\bf affine} and $e\in\co_1$, then $e$ 
is an even nilpotent element and the $\BZ$-grading in question is determined by the semisimple element
$\frac{1}{2}h\in\g(0)$ for an appropriate $\tri$-triple $\{e,h,f\}$.  
Here $(\ad\, e)^3=0$, $\Ima(\ad\,e)^e=\g(1)$, and one can take $D=(\ad\, e)^2$. 
\end{ex}
The setting of Example~\ref{ex:short-Z}(2) can vastly be generalised. 

\begin{prop}           \label{prop:D=e^2}
Let $e\in\q$ be an arbitrary element and $D=(\ad\,e)^2\in {\sf End}(\q)$. Then
\begin{gather}         \label{eq:D=e^2}
\text{$[x,y]'_D=2[[e,x],[e,y]]$ \ for all $x,y\in\q$} \\   \label{eq:n-th}
\text{
and \ $[x,y]^{(n)}_D=2^n{\cdot}[(\ad\,e)^n x,(\ad\,e)^n y]$ \ for all $n\in\BN$ and $x,y\in\q$.}
\end{gather}
\end{prop}
\begin{proof}
Since $\ad\,e:\q\to\q$ is a derivation of $\q$, this follows from Proposition~\ref{prop:1}.
\end{proof}

But Eq.~\eqref{eq:n-th} does not guarantee that $[x,y]^{(n)}_D$ is a Lie bracket on the vector space $\q$.
To obtain a quasi-derivation and thereby compatible Lie brackets, one should impose some constraints on 
$(\q, e)$. Using Eq.~\eqref{eq:n-th} with $n=2$, we obtain

\begin{cl}   \label{cor:D=e^2}
For $D=(\ad\,e)^2$, the following two conditions are equivalent:
\begin{itemize}
\item \  $D$ is a quasi-derivation of $(\q,[\ ,\,])$, i.e., $[\ ,\,]''_D \equiv 0$;
\item \ $[\Ima D,\Ima D]=0$.
\end{itemize}
\end{cl}
\noindent
Therefore, if $D$ is the square of an inner derivation of $\q$, then one has to check only 
condition {\sf (ii)} above. However, if $\q$ is arbitrary, then it is not clear how to check this property.
For this reason, in the rest of the section we stick to the reductive case, where our results are more 
complete. 
We use below standard results on centralisers of nilpotent elements in a reductive Lie algebra $\g$, 
$\tri$-triples, and the corresponding $\BZ$-gradings of $\g$ that can be found in~\cite{CM}.

\begin{thm}                  \label{thm:D=e^2}
If\/ $e\in\g$ and $(\ad\, e)^3=0$, then  
$D=(\ad\, e)^2$ 
is a quasi-derivation of\/ $\g$.
\end{thm}
\begin{proof}
Let $\{e,h,f\}$ be an $\tri$-triple containing $e$, i.e., $[h,e]=2e$, $[e,f]=h$, and $[h,f]=-2f$.  
Consider the $\BZ$-grading $\g=\bigoplus_{i\in\BZ}\g(i)$ determined by $h$, i.e., 
$\g(i)=\{x\in \g\mid [h,x]=ix\}$. Then $e\in\g(2)$ and $f\in\g(-2)$. Since $(\ad\, e)^3=0$, we
have $\g(j)=0$ for $|j|\ge 3$. 
Here $\Ima (\ad\, e)^2=\g(2)$ and $\g(2)$ is abelian. Hence Corollary~\ref{cor:D=e^2} implies that $[\ ,\, ]''_{D}\equiv 0$. 
\\ \indent
(Or, one can notice that $\Ker (\ad\, e)^2=\g({\ge}{-}1)$ and check conditions {\sf (i)}--{\sf (iii)} above.) 
\end{proof}

\begin{rmk}     \label{height-2}
The nilpotent orbits $G{\cdot}e$ with $(\ad\,e)^3=0$ are quite small. In terminology of \cite{p94}, these are 
orbits {\it of height~2}. All such orbits are spherical, and their description can be extracted from the 
classification of spherical nilpotent orbits in~\cite{p94}. For a classical Lie algebra $\g=\g(\gV)$ and a 
nilpotent orbit $G{\cdot}e\subset\g$, let $\blb(e)=(\lb_1,\lb_2,\dots)$ be the corresponding partition of 
$\dim\gV$. If $\g=\slv$ or $\spv$, then $(\ad\,e)^3=0$ if and only if $\lb_1=2$. For $\sov$, $(\ad\,e)^3=0$ if 
and only if either $\lb_1=2$ or $\lb_1=3$ and $\lb_2=1$. 
\end{rmk}
For a quasi-derivation $D\in\gl(\g)$, the pencil $\eus P=\{\ap[\ ,\,]{+}\beta[\ ,\, ]'_D\mid \ap,\beta\in\bbk\}$ 
of compatible Lie brackets has the unique, up to scalar, degenerate operation ($\sim$\,singular line), which is 
$[\ ,\, ]'_D$, see Corollary~\ref{cor:new}. In order to study the \PC\ subalgebra of $\gS(\g)$ corresponding 
to $\eus P$, one has to know the index of $(\g, [\ ,\, ]'_D)$. Below, we keep the symbol $\g$ for the initial 
(reductive) Lie algebra, whereas the Lie algebra $(\g, [\ ,\, ]'_D)$ is denoted by $\g_D$.
That is, $\g_D$ and $\g$ are the same vector spaces equipped with different Lie brackets. Write 
$\z(\g_D)$ for the centre of the Lie algebra $\g_D$. Let $\g^e$ denote the centraliser of $e$ in $\g$.

\begin{thm}             \label{thm:e^2-index}
If\/ $D=(\ad\, e)^2$ and $(\ad\, e)^3=0$, then $\ind(\g, [\ ,\, ]'_D)=\dim\g^e$ and $\z(\g_D)=\g^e$.
\end{thm}
\begin{proof}
For $\eta\in (\g_D)^*$, let $(\g_D)^\eta$ denote the stabiliser of $\eta$ \wrt\ the coadjoint representation
of $\g_D$. It follows from Eq.~\eqref{eq:D=e^2} that  $\g^e\subset\z(\g_D)$.  
Therefore, $\g^e\subset (\g_D)^\xi$ for any $\xi\in (\g_D)^*$ and hence
$\ind \g_D\ge \dim\g^e$. To get the equality, one has to find $\xi\in (\g_D)^*$ such that $(\g_D)^\xi=\g^e$.
\\ \indent  
\textbullet \ 
As above, $\{e,h,f\}$ is an $\tri$-triple and $\g=\bigoplus_{i=-2}^2\g(i)$ is the associated $\BZ$-grading. 
Then  $\g^e=\g(0)^e\oplus\g(1)\oplus\g(2)$. Here $\g(0)^e=\g(0)^f$ is the centraliser in $\g$ of the 
$\tri$-triple and $\g(0)=[f,\g(2)]\oplus \g(0)^e=[e,\g(-2)]\oplus \g(0)^e$.
\\ \indent  
\textbullet \ 
Let $\eus B$ be a non-degenerate symmetric bilinear $\g$-invariant form on $\g$. For instance, one can
take $\eus B$ to be the Killing form if $\g$ is semisimple. Using $\eus B$, we may regard $f$ as an 
element of $(\g_D)^*\simeq\g^*$. Then the {\it Kirillov--Kostant 2-form\/} $\eus K_f$  associated with $f$ is 
given by 
\[
    \eus K_f(x,y)=\eus B([[e,x],[e,y]],f) ,
\]
and the stabiliser $(\g_D)^f$ equals the kernel of $\eus K_f$.  
\\ \indent  
\textbullet \ 
Using the $\g$-invariance of $\eus B$, we obtain $\eus K_f(x,y)=-\eus B(x, [e,[[e,y],f]])$. Hence 
\[
      \Ker(\eus K_f)=\{y\in\g_D \mid [e,[[e,y],f]]=0\} .
\] 
It is known that $\g^e\oplus\Ima(\ad\,f)=\g$ and $\Ima(\ad\,f)=\g(-2)\oplus\g(-1)\oplus\me$, where 
$\me=[f,\g(2)]$. Using  properties of the $\BZ$-grading associated with an $\tri$-triple, one readily verifies 
that the linear map
\[
    \bigl( y\in \Ima(\ad\,f) \bigr) \mapsto \bigl( \tilde y:=[e,[[e,y],f]] \in \Ima(\ad\,e) \bigr)
\]
is bijective. More precisely, if $y\in\g(-2)$, then $\tilde y\in\me$, if $y\in\g(-1)$, then $\tilde y\in\g(1)$, and if 
$y\in\me$, then $\tilde y\in\g(2)$. Thus, $\Ker(\eus K_f)=\g^e$, and we are done.
\end{proof}
\begin{cl}[of the proof]
$\dim \z(\g_D)=\ind\g_D$.
\end{cl}

Lie algebras with this property are sometimes called {\it square integrable}, see~\cite[Def.\,1.8]{ag03}.
\begin{prop}
The Lie algebra $\g_D$ is 2-step nilpotent, i.e., $[\g_D,\g_D]\subset\z(\g_D)$.
\end{prop}
\begin{proof}
As above, the grading $\g=\bigoplus_{i=-2}^2\g(i)$ is associated with $\{e,h,f\}$, $e\in\g(2)$, and 
$[x,y]'_D=2 [[e,x],[e,y]]$.
The last formula shows that $[\g(i),\g(j)]'_D=0$ if $i+j>-2$. For the remaining pairs $(i,j)$, we have:
\\[.5ex]  \centerline{
$[\g(0),\g(-2)]'_D\subset \g(2)$, \quad $[\g(-1),\g(-1)]'_D\subset \g(2)$, \quad $[\g(-1),\g(-2)]'_D\subset \g(1)$,
}   \\[.4ex]
and $[\g(-2),\g(-2)]'_D\subset \g(0)^e$. The first three inclusions easily follow from the properties of 
$\tri$-triples, but the last one requires an explanation. Here $[e,x], [e,y]\subset \me=[e,\g(-2)]$ can be 
arbitrary elements of $\me$,  and we need the property that $[\me,\me]\subset \g(0)^e$, i.e., the sum
$\g(0)=\g(0)^e\oplus\me$ is a $\BZ_2$-grading of $\g(0)$. And this is proved in \cite[Prop.\,3.3]{p94}.
Thus, we have $[\g_D,\g_D]\subset \g^e=\z(\g_D)$.
\end{proof}

\section{Near-derivations versus Nijenhuis operators}
\label{sect:near-Nij}

Let $D\in{\gl}(\q)$ be an arbitrary operator and $[\ ,\,]'_D$ the derived bilinear operation on the vector 
space $\q$. In Section~\ref{sect:near}, we have proved that if $D$ is a near-derivation of the Lie algebra 
$\q$, then $[\ ,\,]'_D$ is a Lie bracket compatible with $[\ ,\,]$. However, there is another interesting class 
of operators $D$ that guarantees similar properties of $[\ ,\,]'_D$, see~\cite{irene,ksm90}.

\begin{df}   \label{def:Nijen}
An operator $\eus N\in {\gl}(\q)={\sf End}(\q)$ is called {\it Nijenhuis}, if 
\\[.6ex] 
\centerline{
$\eus N\bigl([\eus Nx,y]+[x,\eus Ny]-\eus N([x,y])\bigr)=[\eus Nx,\eus Ny]$ for all $x,y\in\q$.}
\end{df}
Using our notation for derived operatioins, this condition can also be written as
\beq     \label{eq:N-prop}
      [\eus Nx,\eus Ny]+\eus N([x,y]'_\eus N)=0 . 
\eeq
The following results are obtained in~\cite[Chap.\,3]{irene}, \cite[Sect.\,1.3]{ksm90}:
\begin{itemize}
\item[\sf (N1)] \  $\eus N^n$ is Nijenhuis for any $n\in \BN$; (Here $\eus N^n$ is the usual $n$-th power of operator.)
\item[\sf (N2)] \  Up to sign, the $m$-th derived operation of $[\ ,\,]$ \wrt\ $\eus N$ coincides with the first
derived operation \wrt\ $\eus N^m$. Namely,  $[\ ,\,]^{(m)}_\eus N=(-1)^{m-1}[\ ,\,]'_{\eus N^m}$
for $m\ge 1$;
\item[\sf (N3)] \  $[\ ,\,]^{(m)}_\eus N$ is a Lie bracket for each $m=0,1,2,\dots$.
\item[\sf (N4)] \  the Lie brackets $[\ ,\,]^{(k)}_\eus N$ and $[\ ,\,]^{(l)}_\eus N$ are compatible for all 
$k,l\in\BN$.
\end{itemize}
This implies that a Nijenhuis operator $\eus N$ provides an $m$-parameter family of compatible Lie 
brackets on $\q$ (=\,compatible linear Poisson brackets in $\gS(\q)$), where $m$ is the degree of the 
minimal polynomial of $\eus N\in{\gl}(\q)$.

For $\eus D\in\gl(\q)$, the {\it Nijenhuis torsion} of $\eus D$ is the cochain 
$\mathcal T_{\eus D}\in\gC^2(\q,\q)$ defined by 
\[
  \ct_{\eus D}(x,y):=[\eus Dx,\eus Dy]+  \eus D\bigl(\eus D([x,y])-[\eus Dx,y]-[x,\eus Dy]\bigr).
\] 
Hence $\eus D$ is Nijenhuis $\Leftrightarrow$ $\mathcal T_\eus D=0$. It is known that $[\ ,\,]'_{\eus D}$ is 
a Lie bracket if and only if $\ct_{\eus D}$ is a $2$-cocycle with values in $\q$, i.e., 
$\ct_{\eus d}\in \eus Z^2(\q,\q)$,
see~\cite[\S\,1]{gs02}, \cite[Prop.\,1.1]{ksm90}. 

\begin{rmk}        \label{rem:iterated-vs}
For the Nijenhuis operators, it is convenient to change sign in~\eqref{eq:derived-Lie} and deal with the 
modified derived operation
\[  
    (x,y)\mapsto [x ,y]'_{(\eus N)}:=-[x ,y]'_\eus N=[\eus Nx,y]+[x,\eus Ny]-\eus N([x,y] . 
\] 
Exactly this bracket and its higher iterations are considered in \cite{irene,gs02,ksm90}. This simplifies 
some formulas, e.g. then {\sf (N2)} asserts that $[\ ,\,]^{(m)}_{(\eus N)}=[\ ,\,]'_{(\eus N^m)}$.
However, in this case, we loose the connection with the representations $\rho: \gl(\q)\to \Pi_2^1(\q)$ and
$\cR: GL(\q)\to \Pi_2^1(\q)$, which is required in Theorem~\ref{thm:4a} below.
\end{rmk}

\begin{rmk}                      \label{rem:versus}
If $\eus N\in {\sf End}(\q)$ is a Nijenhuis operator, then it follows from {\sf (N2)} that  

{\bf --} \ $\eus N$ is a near-derivation if and only if the degree of its minimal polynomial equals $2$;
\\ \indent
{\bf --} \ $\eus N$ is a quasi-derivation if and only if $\eus N^2=0$. 
\\
On the other hand, there are semisimple near-derivations of $\q$ that are not Nijenhuis. If $\eus N$ is 
Nijenhuis and semisimple, then its eigenspaces are subalgebras of $\q$~\cite[Prop.\,2]{gs02}, 
\cite[Prop.\,1.4]{panas}. Therefore, the near-derivations in Sections~\ref{subs:period} and 
\ref{subs:ext-period} are not Nijenhuis, whereas $D$ of Section~\ref{subs:2-split} is both a near-derivation
and Nijenhuis.
\end{rmk}

Next, we point out a connection between the Nijenhuis torsion and derived operations.
\begin{prop}      \label{prop:helpful-torsion}
For any $\eus D\in\gl(\q)$, one has
$[x,y]''_{\eus D}=2\ct_{\eus D}(x,y)- [x,y]'_{\eus D^2}$.
\end{prop}   
\begin{proof}
Compare both parts using the definitions.
\end{proof}
This readily implies the following properties.
\begin{itemize}
\item \ If $\eus D^2=0$, then $[x,y]''_{\eus D}=2\ct_{\eus D}(x,y)$. In particular, 
$\eus D$ is a quasi-derivation if and only if $\eus D$ is Nijenhuis.
\item \ If $\ct_{\eus D}(x,y)\equiv 0$, i.e., $\eus D$ is Nijenhuis, then $[x,y]''_{\eus D}=- [x,y]'_{\eus D^2}$
(cf. {\sf (N2)} above).
\end{itemize}

\begin{rmk}                     \label{rem:advantage}
A useful property of a near-derivation $D\in\gl(\q)$ is that using its extension to $\gS(\q)$ and $\cz\gS(\q)$, 
one obtains a \PC\ subalgebra $\cz_D\subset\gS(\q)$, cf. Theorem~\ref{thm:near3}({\sf ii}). In many 
cases, that procedure provides a \PC\ subalgebra of maximal possible transcendence degree. This 
concerns the cases considered in Sections \ref{subs:period} and \ref{subs:2-split}, if $\q$ is reductive 
(see~\cite{period1} and \cite{bn}, respectively). The new case of Section~\ref{subs:ext-period} will be 
treated elsewhere. However, there is no direct analogue of Theorem~\ref{thm:near3}({\sf ii}) for the 
extension of Nijenhuis operators to $\gS(\q)$. Still, we have a curious property of the Nijenhuis operators 
that resembles Theorem~\ref{thm:near2}, with $\q$ in place of $\gS(\q)$. If $\eus N$ is Nijenhuis, then 
Eq.~(3.8) in~\cite{irene} implies the following for $x,y\in\q$.
\\[.5ex]
\centerline{
\emph{Suppose that\/ $[x,y]=[\eus N^i x,y]=[x,\eus N^j y]=0$ \ for some $i,j\in\BN$. Then 
$[\eus N^ix,\eus N^j y]=0$.}} 
\end{rmk}

As before, $\exp(M)$ or $e^{M}$ denotes the usual exponent of a matrix $M\in{\gl}(\q)$. 
\begin{thm}      \label{thm:4a}
If\/ $\eus N\in{\gl}(\q)$ is Nijenhuis, then 
\beq        \label{eq:4a}
    e^{-s\eus N}([e^{s\eus N}x,e^{s\eus N}y])=[e^{s\eus N}x,y]+[x,e^{s\eus N}y]- e^{s\eus N}([x,y]) .
\eeq
for all $x,y\in\q$ and $s\in\bbk$.
\end{thm}
\begin{proof}
Let $\sT=\sT_\q\in\Pi_2^1(\q)$ be the tensor of structure constants of the Lie algebra $\q$. Recall from 
Section~\ref{sect:near} that $\cR: GL(\q)\to GL(\Pi_2^1(\q))$ is the natural representation
in the space of tensors of type $(1,2)$ on $\q$ and
\beq         \label{eq:R-summa}
    \cR(\exp(-s\eus N)){\cdot}\sT=\sum_{m\ge 0}\frac{(-s)^m}{m!}\rho(\eus N)^m{\cdot}\sT.
\eeq
In terms of $\q$, the LHS represents the Lie bracket  
$(x,y)\mapsto e^{-s\eus N}([e^{s\eus N}x,e^{s\eus N}y])$. In the RHS, we use the relation 
$\rho(\eus N)^m{\cdot}\sT=(-1)^{m-1}\rho(\eus N^m){\cdot}\sT$, which is another form of {\sf (N2)}. 
Hence the RHS equals
\beq   \label{eq:sprava}
  \sT-\sum_{m\ge 1}\frac{s^m}{m!}\rho(\eus N^m){\cdot}\sT \ .
\eeq
Since the tensor $\rho(\eus N^m){\cdot}\sT$ represents the operation 
\[
     (x,y)\mapsto \eus N^m([x,y]) - [\eus N^m x,y]-[x,\eus N^my], 
\]
the whole expression~\eqref{eq:sprava} corresponds to the operation
\begin{multline*}
(x,y)\mapsto [x,y]- (e^{sN}{-}I)[x,y]+[(e^{sN}{-}I)x,y]+[x, (e^{sN}{-}I)y]=
\\  =[e^{s\eus N}x,y]+[x,e^{s\eus N}y]- e^{s\eus N}([x,y]) .  \qedhere
\end{multline*}
\end{proof}

\begin{cl} 
If\/ $\eus N$ is Nijenhuis, then so is $e^{s\eus N}$ ($s\in\bbk$).  
\end{cl}
Comparing the coefficients of $s^n$ in \eqref{eq:4a} yields the equality
\[
   \sum_{k+i+j=n}\frac{1}{k! i! j!}(-1)^k \eus N^k([\eus N^ix, \eus N^jy])=\frac{1}{n!}\bigl([\eus N^nx,y]+
   [x,\eus N^ny]-\eus N^n([x,y])\bigr).
\]
If $n=0,1$, then it is a tautology, while for $n=2$ one obtains the defining relation of the Nijenhuis 
operators, see Def.~\ref{def:Nijen}. For $n\ge 3$, one obtains certain relations that are formal consequences 
of~\eqref{eq:N-prop}. For instance, if $n=3$, then we get, after some transformations,
\\[.5ex]
\centerline{$[\eus Nx,\eus Ny]'_\eus N=\eus N([x,y]'_{\eus N^2})=-\eus N([x,y]''_\eus N)$. }

A relation similar to~\eqref{eq:4a} can be obtained for the near-derivations.
\begin{thm}            \label{thm:exp-near}
Let\/ $D$ is a near-derivation of\/ $\q$ such that $\rho(D)\vert_\BU=
{\small \begin{pmatrix} 0 & 0 \\ 1 & m \end{pmatrix}}$, i.e., 
$\rho(D){\cdot}\sT=\sT'$ and $\rho(D){\cdot}\sT'=m\sT'$. Then 
\beq        \label{eq:4b}
      e^{sD }([e^{-sD}x,e^{-sD}y])=[x,y]+\frac{1}{m}(e^{ms}-1)[x,y]'_D .
\eeq
for all $x,y\in\q$ and $s\in\bbk$.
\end{thm}
\begin{proof}
We begin with the equality
\[
   \cR(\exp(sD)){\cdot}\sT=\sum_{n\ge 0}\frac{s^n}{n!}\rho(D)^n{\cdot}\sT.
\]
Here the LHS represents the Lie bracket $(x,y)\mapsto e^{sD}([e^{-sD}x,e^{-sD}y])$.
In the RHS, we use the relation $\rho(D)^n{\cdot}\sT=m^{n-1}\rho(D){\cdot}\sT$. This yields the sum
\[
    \sT+ (s+\frac{ms^2}{2!}+\frac{m^2s^3}{3!}+\dots)\rho(D){\cdot}\sT=
    \sT+\frac{1}{m}(e^{ms}-1)\rho(D){\cdot}\sT ,
\]
which represents the operation in the RHS of Eq.~\eqref{eq:4b}. 
\end{proof}

{\bf Remark.} \ This result is also explained by the relation \ 
\[
\exp(\rho(sD)\vert_\BU)=
\exp\!{\small \begin{pmatrix} 0 & 0 \\ s & sm\end{pmatrix}
=\begin{pmatrix}
1 & 0 \\ \frac{1}{m}(e^{ms}{-}1) & e^{ms}  \end{pmatrix}}, 
\]
which implies that $\cR(\exp(sD)){\cdot}\sT=\sT+\frac{1}{m}(e^{ms}-1)\sT'$.

If $m=0$ in~Theorem~\ref{thm:exp-near}, then $D$ is a quasi-derivation of $\q$, 
$\cR(\exp(sD ))\vert_\BU={\small \begin{pmatrix}
1 & 0 \\ s & 1  \end{pmatrix}}$,
and one gets the 
relation $e^{sD}([e^{-sD}x,e^{-sD}y])=[x,y]+s[x,y]'_D$. 

\section{Associative algebras and Nijenhuis operators} 
\label{sect:assoc}
Here we recall a well-known approach to Nijenhuis operators via associative algebras, provide some 
minor complements, and relate this to near-derivations.
Let $(\ca, {\cdot})$ be a finite-dimensional associative algebra and $(\ca, [\ ,\,])$ the corresponding 
Lie algebra, i.e., $[x,y]=xy-yx$ for $x,y\in\ca$.

{\sf\bfseries (1)}\quad  For $a\in\ca$, consider the operator $L_a: \ca\to \ca$ such that $L_a(x)=ax$. 
Then $(L_a)^n=L_{a^n}$. It is well known (and easily verified) that $L_a$ is Nijenhuis for $(\ca, [\ ,\,])$, 
cf.~\cite[Ex.\,1]{gs02}, \cite[Ex.\,1.5]{panas}. This yields the derived (compatible) Lie bracket
\[
      [x,y]'_{L_a}=a[x,y]-[ax,y]-[x,ay]=-xay+yax. 
\]
and the relation $[\ ,\,]^{(n)}_{L_a}=(-1)^{n-1}[\ ,\,]'_{L_a^n}$.
In the spirit of Remark~\ref{rem:iterated-vs}, it is convenient to change sign and consider the Lie bracket
\beq       \label{eq:a-bracket}
    [x,y]_a:=-[x,y]'_{L_a}=xay-yax.
\eeq
Here we have $[x,y]''_{L_a}=-[x,y]'_{L_a^2}=[x,y]_{a^2}$. Hence $L_a$ is a quasi-derivation of 
$(\ca, [\ ,\,])$ whenever $a^2=0$. (This fact is noticed by Vinberg in~\cite{vi95}.)  If $(\ca, {\cdot})$ has the 
unit ${\bf i}$, then $L_a$ is a near-derivation whenever $a^2\in\lg a,{\bf i}\rg$.

{\sf\bfseries  (2)}\quad  This can be extended to some related Lie algebras as follows, cf.~\cite[Ex.\,2]{gs02}. 
Let `$\ast$' be an involution of $\ca$ such that $(xy)^*=y^*x^*$. Set 
\[
  \ca_0=\{x\in\ca\mid x^*=-x\} \ \text{ and } \ \ca_1=\{x\in\ca\mid x^*=x\}. 
\]
The sum $\ca=\ca_0\oplus\ca_1$ is a $\BZ_2$-grading of the Lie algebra $\ca$, the corresponding
involution $\sigma$ of $(\ca, [\ ,\,])$ being $x\mapsto \sigma(x)=-x^*$. In particular,
$\ca_0$ is a Lie subalgebra of  $(\ca, [\ ,\,])$. 
If $a\in\ca_1$, then $\ca_0$ is also a Lie subalgebra of
$(\ca, [\ ,\,]_a)$. This yields compatible Lie brackets $[\ ,\,]$ and $[\ ,\,]_a$ on the vector space $\ca_0$.
\\ \indent
In general, $L_a$ with $a\in\ca_1$ does not preserve $\ca_0$. Hence one cannot immediately claim that 
$[\ ,\,]_a$ is a derived operation for $(\ca_0, [\ ,\,])$ \wrt\ some element of ${\sf End}(\ca_0)$. A remedy is 
to replace $L_a$ with the operator $D_a{:=}\frac{1}{2}(L_a+R_a)$, where $R_a(x)=xa$. 
A straightforward verification shows that $D_a$ preserves $\ca_0$ and 
\[
        -[x,y]'_{D_a}=-[x,y]'_{L_a}=[x,y]'_a=xay-yax
\]
for $x,y\in\ca_0$. Hence $[\ ,\,]_a$ is the derived operation in $\ca_0$ \wrt\ $-D_a$. In general, the 
operator $D_a\in\gl(\ca_0)$ is not Nijenhuis and property {\sf (N2)} does not hold with $n=2$. 
Nevertheless, a direct calculation shows that $[x,y]''_{D_a}=[x,y]_{a^2}$ for $x,y\in\ca_0$ and $a\in\ca_1$. 
Since $[L_a,\ad\,a]=0$ and $D_a=L_a-\frac{1}{2}\ad\,a$, this also follows from Proposition~\ref{prop:2}.
This means again that  $D_a$ is a quasi-derivation (resp. near-derivation) of $\ca_0$ whenever 
$a^2\vert_{\ca_0}=0$ (resp. $a^2\in\lg a,{\bf i}\rg$). Note that if $a\in\ca_1$, then $a^2\in\ca_1$ as well.
Thus, if $a^2\in\lg a,{\bf i}\rg$, then we can use $D_a$ for obtaining the \PC\ subalgebra of $\gS(\ca_0)$ 
associated with pencil $\ap[\ ,\,]+\beta[\ ,\,]_a$.

Although the definition of $D_a$ exploits the associative structure of $\ca$, it appears, surprisingly, that 
the Nijenhuis torsion of $D_a: \ca_0\to \ca_0$ ($a\in\ca_1$) can be expressed solely in terms of the Lie 
bracket on $\ca$.

\begin{prop}    \label{prop:ass-torsion}
For $x,y\in\ca_0$ and $a\in\ca_!$, we have 
$\displaystyle\mathcal T_{D_a}(x,y)=-\frac{1}{4}{\cdot}[[a,x],[a,y]]$. 
\end{prop}
\begin{proof}
Using the definition of $\mathcal T_{D_a}$ and $[\ ,\,]$, expand both parts and compare results.  
\end{proof}
By the Jacobi identity, the RHS in Proposition~\ref{prop:ass-torsion} can be written differently. Namely, in any Lie algebra one has
\[
  2 [[a,x],[a,y]]=[a,[a,[x,y]]]-[[a,[a,x]],y]-[x,[a,[a,y]]]=[x,y]'_{(\ad\,a)^2} ,
\]
cf. Proposition~\ref{prop:D=e^2}.
This shows that if $(\ad\,a)^2\vert_{\ca_0}=0$, then $\mathcal T_{D_a}=0$ and $D_a$ is Nijenhuis. Note
that $\ad\,a$ takes $\ca_0$ to $\ca_1$, but $(\ad\,a)^2$ preserves $\ca_0$.

{\sf\bfseries  (3)}\quad Using suitable involutions of the full matrix algebra $\ca=\gln$, one obtains various 
near- and  quasi-derivations of the Lie algebras $\son$ and $\mathfrak{sp}_{2m}$, as follows.  

For any $J\in GL_n$, the map $(x\in\gln)\mapsto x^*:=J^{-1}x^t J$ has the property that $(xy)^*=y^*x^*$. 
If  $J$ is either symmetric or skew-symmetric, then $(x^*)^*=x$ and $\ca_0$ is either an orthogonal or a 
symplectic Lie algebra. (Of course, $n$ must be even, if $J$ is skew-symmetric.)

However, the approach via associative algebras does not work for $\sln$, not to mention the exceptional 
Lie algebras.


\end{document}